\documentclass[11pt,a4paper]{article}
\usepackage[utf8]{inputenc}
\usepackage{amsmath,amssymb,amsthm}
\usepackage{geometry}
\usepackage{hyperref}
\usepackage{booktabs}
\usepackage{microtype}
\usepackage{fancyhdr}
\usepackage{enumitem}
\usepackage{times}
\usepackage{authblk}

\hypersetup{
    colorlinks=true,
    linkcolor=blue,
    citecolor=blue,
    filecolor=magenta,      
    urlcolor=cyan,
    pdftitle={Existence of Admissible Subsolutions to the Dirichlet Problem for Symmetric Augmented k-Hessian Type Equations in Bounded Domains},
}

\newtheorem{theorem}{Theorem}[section]
\newtheorem{lemma}[theorem]{Lemma}
\newtheorem{proposition}[theorem]{Proposition}

\theoremstyle{definition}
\newtheorem{definition}[theorem]{Definition}

\newtheorem{example}[theorem]{Example}

\numberwithin{equation}{section}

\fancypagestyle{firstpage}{
  \fancyhf{}
  \fancyfoot[L]{\itshape Preprint}
  \fancyfoot[R]{\today}

}

\title{\textbf{Existence of Admissible Subsolutions to the Dirichlet Problem for Symmetric Augmented $k$-Hessian Type Equations in Bounded Domains}\footnote{This research is funded by Hanoi Pedagogical University 2 under grant number HPU2.2022-UT-08}}

\author[1]{Dinh Hong Quang}
\author[2]{Tran Van Bang\thanks{Corresponding author. E-mail: tranvanbang@hpu2.edu.vn}}
\author[3]{Ha Tien Ngoan}
\author[4]{Nguyen Huu Tho}
\author[5]{Phan Trong Tien}

\affil[1]{PhD student at Faculty of Mathematics, Hanoi Pedagogical University 2, Phutho, Vietnam}
\affil[2]{Faculty of Mathematics, Hanoi Pedagogical University 2, Phutho, Vietnam}
\affil[3]{Faculty of Applied Science, University of Economics - Technology for Industries, Hanoi, Vietnam}
\affil[4]{Adjunct Lecturer, Faculty of Mathematics and Infomatics, Hanoi University of Science and Technology, Hanoi, Vietnam}
\affil[5]{Department of Mathematics, Quang Binh University,  Quangbinh, Vietnam}

\date{}

\begin{document}

\maketitle
\thispagestyle{firstpage}

\begin{abstract}
We prove the existence of admissible subsolutions to the Dirichlet problem for symmetric augmented $k$-Hessian type equations. An important sufficient condition is the uniform $(k-1)$-$A$-convexity of the domain $\Omega,$  where $A(x, z, p)$ is the augmented symmetric matrix appearing in the equation. This condition was originally introduced by F. Jiang, N. S. Trudinger, and X.-P. Yang and we have chosen a special their case. The structural conditions on the matrix $A(x, z, p)$ include its growth with respect to the variables $z$ and $p,$  particularly requiring that some of its first and second derivatives are sufficiently small in a sufficiently small neighborhood of the boundary. Under certain structural conditions on $A(x, z, p),$  the uniform $(k-1)$-$A$-convexity of $\Omega$ is also a necessary condition for the existence of admissible subsolutions of the equation in a neighborhood of the boundary. Our results extend the classic result by L. Caffarelli, L. Nirenberg, and J. Spruck from the case $A \equiv 0$ to the general case $A \neq 0.$ Our same theorems are valid also for augmented quotient Hessian type equations.
\end{abstract}

\section{Introduction}

This paper studies the existence of admissible subsolutions to the Dirichlet problem for the following symmetric augmented $k$-Hessian type equation:
\begin{equation}\label{eq1.1}
S_k( D^2 u - A(x, u, Du) ) = f(x, u, Du) \quad \text{in } \Omega,
\end{equation}
\begin{equation}\label{eq1.2}
u(x) = \phi(x) \quad \text{on } \partial\Omega,
\end{equation}
where $\Omega$ is a bounded domain in $\mathbb{R}^n$ with connected boundary $\partial\Omega \in C^2,$  $A(x, z, p) = [A_{ij}(x, z, p)]_{n \times n} \in \mathbb{R}^{n \times n}$ is a symmetric matrix ($A^T = A$), $f(x, z, p)$ is a positive scalar-valued function, and $\phi(x)$ is a given function on $\partial\Omega.$ 

The operator $S_k$ is defined by
\[
S_k(M) = \sigma_k(\lambda(M)),
\]
where $M$ is an $n \times n$ symmetric matrix, $\lambda(M) = (\lambda_1, \dots, \lambda_n) \in \mathbb{R}^n$ is the vector of eigenvalues of $M,$  and $\sigma_k(\lambda)$ is the $k$-th elementary symmetric polynomial:
\[
\sigma_k(\lambda) = \sum_{1 \leq i_1 < \dots < i_k \leq n} \lambda_{i_1} \dots \lambda_{i_k}.
\]
We denote by $\widetilde{\Gamma}_k$ the open cone
\[
\widetilde{\Gamma}_k = \{ \lambda \in \mathbb{R}^n: \sigma_k(\lambda) > 0\},
\]
and by $\Gamma_k$ its connected component, which contains the vector $(1, 1, \dots, 1).$ Note that $\Gamma_k\subset\Gamma_m, 1\leq m<k,$ and
\[
\Gamma_k = \{ \lambda \in \mathbb{R}^n: \sigma_j(\lambda) > 0,\quad 1\leq j\leq k\}.
\]

For $u(x) \in C^2(\bar{\Omega}),$  we denote:
\begin{equation}\label{eq1.3}
\omega(u)(x) = D^2 u(x) - A(x, u(x), Du(x)).
\end{equation}
A function $u(x)$ is said to be admissible for equation \eqref{eq1.1} in $\Omega' \subset \Omega$ if:
\[
\lambda(\omega(u)(x)) \in \Gamma_k \quad \text{for all } x \in \Omega'.
\]
If, in addition, $u(x)$ is a solution (subsolution) of equation \eqref{eq1.1} in $\Omega',$  then $u(x)$ is called an admissible solution (admissible subsolution) of equation \eqref{eq1.1} in $\Omega'.$  If $\Omega' = \Omega$ and $u(x) = \phi(x) $ on $\partial\Omega,$  then $u(x)$ is called an admissible solution (admissible subsolution) of the problem \eqref{eq1.1}, \eqref{eq1.2}.

When $A(x, z, p) \equiv 0,$  equation \eqref{eq1.1} is the standard $k$-Hessian equation; when $A(x, z, p) \neq 0,$  \eqref{eq1.1} is an augmented $k$-Hessian type equation; and when $k = n,$  \eqref{eq1.1} becomes the augmented Monge-Ampère equation. Currently, there are many papers studying the solvability of the problem \eqref{eq1.1}, \eqref{eq1.2}, such as \cite{3, 8, 7} and others, even for the class of augmented nonsymmetric $k$-Hessian type equations \cite{2}. To obtain a priori estimates on the solution in $C^1$ and $C^2,$  the authors of the aforementioned papers all make the same assumption, namely, the existence of an admissible subsolution to problem \eqref{eq1.1}, \eqref{eq1.2}. This paper of ours aims to introduce a class of matrices $A(x, z, p)$ to provide an affirmative answer regarding the existence of such admissible subsolutions.

In the case $A(x, z, p) \equiv 0,$  in the seminal 1985 paper, L. Caffarelli, L. Nirenberg, and J. Spruck \cite{3} proved the existence of an admissible subsolution of the problem \eqref{eq1.1}, \eqref{eq1.2} under the assumption that the boundary $\partial\Omega$ is uniformly $(k-1)$-convex. They proved the existence of admissible subsolutions even for a wider class of equations where the polynomial $\sigma_k(\lambda)$ is replaced by any symmetric hyperbolic polynomial satisfying an additional boundary condition on $\partial\Omega,$  where $\partial\Omega$ is uniformly $(k-1)$-convex in the sense associated with this polynomial.

In this paper, we apply the techniques from \cite{3} to extend the results to the class of equations \eqref{eq1.1} with $A(x, z, p) \neq 0,$  where the uniform $(k-1)$-convexity condition of the domain $\Omega$ is replaced by the uniform $(k-1)$-$A$-convexity of $\Omega.$  The uniform $(k-1)$-$A$-convexity of the domain $\Omega$ was proposed by F. Jiang, N. S. Trudinger, and X.-P. Yang in \cite{7, 8} during their a priori second derivative estimates of admissible solutions. However, the concept of uniform $(k-1)$-$A$-convexity introduced by those authors is not an independent geometric property of the domain $\Omega$ itself, but rather depends on some non-fixed function $u(x)\in C^2(\overline{\Omega}).$ Below we have chosen a special their case by setting $u(x)\equiv 0.$

Suppose equation \eqref{eq1.1} is written in the orthonormal coordinate system $O e_1 \dots e_n.$  Denote by $\nu(x_0) = (\nu_1(x_0), \dots, \nu_n(x_0))$ the unit outward normal vector at $x_0 \in \partial\Omega.$  Suppose $e'_1, \dots, e'_n$ is an orthonormal basis such that $e'_n = -\nu(x_0),$  and the new origin $O'$ is shifted to $x_0.$  Let $J(x_0) = [J_{ij}(x_0)]_{n \times n}$ be the transformation matrix from $e_1, \dots, e_n$ to $e'_1, \dots, e'_n,$  and let $\tilde{J}(x_0) = [J_{ij}(x_0)]_{n \times (n-1)}$ be its submatrix.

In the coordinate system $O' e'_1 \dots e'_n,$  the boundary $\partial\Omega$ is locally represented by
\[
y_n = \rho(y'), \quad y' = (y_1, \dots, y_{n-1}).
\]
Let $\kappa_1(x_0), \dots, \kappa_{n-1}(x_0)$ be the principal curvatures of $\partial\Omega$ at $x_0 \in \partial\Omega,$  i.e.,
\[
\lambda(D^2_{y'} \rho(y'))\Big|_{y'=0} = (\kappa_1(x_0), \dots, \kappa_{n-1}(x_0))
\]
(with respect to the inner normal vector $-\nu(x_0)$).
Recall that the domain $\Omega$ is called uniformly $(k-1)$-convex if
\[
(\kappa_1(x_0), \dots, \kappa_{n-1}(x_0)) \in \Gamma_{k-1} \subset \mathbb{R}^{n-1} \quad \text{for all } x_0 \in \partial\Omega,
\]
meaning
\[
\inf_{x_0 \in \partial\Omega} \sigma_{k-1}(\kappa_1(x_0), \dots, \kappa_{n-1}(x_0)) > 0.
\]
We denote:
\begin{equation}\label{eq1.4}
M(x_0, A):= \operatorname{diag}(\kappa_1(x_0), \dots, \kappa_{n-1}(x_0)) + \tilde{J}(x_0)^T \left[ \sum_{j=1}^n D_{p_j} A(x_0, 0, 0) \nu_j(x_0) \right] \tilde{J}(x_0).
\end{equation}

\begin{definition}\label{def1.1}
The domain $\Omega$ is called uniformly $(k-1)$-$A$-convex if there exists a constant $C_1 > 0$ such that
\begin{equation}\label{eq1.5}
\inf_{x_0 \in \partial\Omega} S_{k-1}(M(x_0, A)) \geq C_1,
\end{equation}
which is equivalent to $\lambda(M(x_0, A)) \in \Gamma_{k-1} \subset \mathbb{R}^{n-1}$ for all $x_0 \in \partial\Omega.$ 
\end{definition}

One of the main results of this paper is the following theorem on the necessary condition.

\begin{theorem}\label{thm1.1}
Suppose there exists an admissible function $v(x) \in C^2(\bar{\Omega}_{\delta})$ of equation \eqref{eq1.1} in a neighborhood $\Omega_{\delta}$ of the boundary $\partial\Omega,$  where
\[
\Omega_{\delta} = \{ x \in \Omega: d(x) < \delta \}, \quad 0 < \delta \ll 1,
\]
with $d(x) = \operatorname{dist}(x, \partial\Omega)$ denoting the distance from $x$ to $\partial\Omega.$  Suppose that the following conditions are satisfied:
\begin{enumerate}
\item $v(x_0) = 0$ for all $x_0 \in \partial\Omega, v(x)<0$ for all $x\in\Omega_{\delta};$ 
\item $\inf_{x_0 \in \partial\Omega} |Dv(x_0)| > 0;$ 
\item $A(x_0, 0, 0) \geq 0;$ 
\item $\sup_{x_0 \in \partial\Omega} \left[ |Dv(x_0)|^2 \sup_{|p| \leq |Dv(x_0)|} \left[ \sum_{i,j=1}^n |D^2_{p_i p_j} A(x_0, 0, p)|^2 \right]^{1/2} \right] \leq \frac{1}{2n} \inf_{x_0 \in \partial\Omega} [ S_k(\omega(v)(x_0)) \cdot t_k(\omega(v)(x_0)) ],$  \label{eq1.7}
\end{enumerate}
where $\omega(v)(x)$ is defined by \eqref{eq1.3}, and $t_k(\omega(v)(x_0))$ is defined as the (unique) positive solution of the equation \eqref{eq2.19}.
Then $\Omega$ is uniformly $(k-1)$-$A$-convex.
\end{theorem}

Suppose $\Omega' \subset \Omega$ and $g(x, z, p)$ is a non-negative scalar-valued function defined on $\bar{\Omega} \times \mathbb{R} \times \mathbb{R}^n.$  We introduce the following definitions. Let $0 < \chi < 1,$  $0 < \gamma \leq 1,$  $C > 0.$  

\begin{definition}\label{def1.2}
Let $0 < \chi^* < 1,$  $0 < \gamma^* \leq 1,$  $C^* > 0.$ We say that the function $g(x, z, p)$ for $z\leq 0$ has growth order $(C^*, -\chi^*, \gamma^*)$ in $\Omega'$ if
\begin{equation}\label{eq1.8}
\limsup_{|z| \to \infty, z \leq 0, |p| \to \infty} \left[ \sup_{x \in \Omega'} \frac{g(x, z, p) (1 + |z|)^{\chi^*}}{(1 + |p|)^{\gamma^*}} \right] \leq C^*.
\end{equation}
\end{definition}

\begin{definition}\label{def1.3}
	We say that the matrix $A_1(x, z, p)$ satisfies in $\Omega$ the condition $(\mathcal{C}_{\delta})$ if for a sufficiently small $\delta > 0$ the following conditions are fulfilled:
\begin{enumerate}
	\item In the domain $\bar{\Omega} \backslash \Omega_{\delta},$  for $z \leq 0,$  $|A_1(x, z, p)|$ has growth order $(C_0, -\chi, \gamma),$  where
	\begin{equation}\label{eq1.10}
		0 < \chi< 1, \quad 0 < \gamma \leq 1, \quad C_0>0;
	\end{equation}
	\item In the domain $\bar{\Omega}_{\delta},$  for $z \leq 0:$ 
	\begin{enumerate}
		\item $D_z A_1(x, z, 0) \geq 0;$  \label{eq1.11}
		\item $\left[ \sum_{j=1}^n |D_{p_j} A_1(x, z, p)|^2 \right]^{1/2}$ has growth order
		\begin{equation}\label{eq1.12}
			\left( \frac{s_k}{11} \min(1, \delta^{1-2\chi }, \delta^{3\gamma-2}), -\chi, \gamma-1 \right);
		\end{equation}
		\item $\left[ \sum_{i,j=1}^n |D^2_{x_i p_j} A_1(x, z, p)|^2 \right]^{1/2}$ has growth order
		\begin{equation}\label{eq1.13}
			( C_1, -\chi, \gamma-1 );
		\end{equation}
		\item $\left[ \sum_{j=1}^n |D^2_{z p_j} A_1(x, z, p)|^2 \right]^{1/2}$ has growth order
		\begin{equation}\label{eq1.14}
			\left( \frac{s_k}{11} \min(1, \delta^{3\gamma-1}, \delta^{2(1-\alpha_1)}), -\chi-1, \gamma-1 \right);
		\end{equation}
		\item $\left[ \sum_{i,j=1}^n |D^2_{p_i p_j} A_2(x, z, p)|^2 \right]^{1/2}$ has growth order
		\begin{equation}\label{eq1.15}
			\left( \frac{s_k}{11} \min(1, \delta^{3-2\chi }, \delta^{3(1-\alpha_2)}), -\chi, \gamma-2 \right),
		\end{equation}
		where $0 < \alpha_j < 1$ (with $j=1,2$), and $s_k$ is the (unique) positive solution of the equation \eqref{eq2.19};
	\end{enumerate}
\end{enumerate}
\end{definition}

The main sufficient condition result of this paper is the following theorem.

\begin{theorem}\label{thm1.2}
Suppose that the following conditions are satisfied:
\begin{enumerate}
\item $\partial\Omega \in C^2,$  $\phi(x) \in C^2(\partial\Omega);$ 
\item $A(x, z, p) = A_1(x, z, p) + A_2(x, z, p),$  where $A_1$ and $A_2$ are symmetric matrices ($A_1^T = A_1 \in C^2, A_2^T = A_2 \in C^2$) and
\begin{equation}\label{eq1.9}
A_2(x, z, p) \leq 0 \quad \text{for } z \leq 0;
\end{equation}
\item The domain $\Omega$ is uniformly $(k-1)$-$A_1$-convex;
\item There exists $0<\delta_0<<1$ such that the matrix $A_1(x,z,p)$ satisfies in $\Omega$ the condition $(\mathcal{C}_{\delta_0}).$ 
\end{enumerate}

Then there exists an admissible function to the equation \eqref{eq1.1}. Moreover, if the right hand side $f(x,z,p)>0$ for $z\leq 0$ has growth order $(C-2,0,h)$ with $0<h<k,$ then there exists an admissible subsolution to the problem \eqref{eq1.1}, \eqref{eq1.2}.
\end{theorem}

The structure of the paper is as follows. In Section 2, a proposition regarding the sum of two matrices is proved, showing that if one matrix has its vector of eigenvalues in $\Gamma_k$ and the other matrix is ``sufficiently small'' in a certain sense relative to it, then the vector of eigenvalues of their sum also belong to $\Gamma_k.$  This proposition plays a decisive role in proving Theorems \ref{thm1.1} and \ref{thm1.2}. In Section 3, Theorem \ref{thm1.1} on the necessary condition is proved. In Section 4, Theorem \ref{thm1.2} is proved by explicitly constructing an admissible function and an admissible subsolution through several steps. We apply the construction of the subsolution as developed in the aforementioned paper \cite{3}. In addition to depending on a sufficiently large positive parameter $a,$  our admissible function also depends on five sufficiently small and fixed positive parameters: $\gamma_j,$  $j = 1, 2, \dots, 5.$  To apply the proposition on the sum of two matrices, we perform estimates on several matrices, primarily within a sufficiently small boundary neighborhood $\Omega_{\delta}.$  In Section 5, we present illustrative examples with concrete matrices $A(x, z, p).$

\section{Some Auxiliary Propositions}

\subsection{Lagrange type Formula for Matrix Functions}
In this section, we utilize a Lagrange-type formula for multi-variable matrix-valued functions, which plays a role analogous to the Newton-Leibniz formula but is adapted for functions taking values in the space of square matrices.

\begin{proposition}\label{prop2.1}
Suppose $F(W) = F(W_1, \dots, W_m) \in C^1(\mathcal{W}, \mathbb{R}^{n \times n}),$  where $\mathcal{W}$ is a convex domain in $\mathbb{R}^m.$  Then for any $W^{(0)}, W^{(1)} \in \mathcal{W},$  we have:
\begin{equation}\label{eq2.1}
F(W^{(1)}) - F(W^{(0)}) = \sum_{j=1}^m \int_0^1 F_{W_j}(tW^{(1)} + (1-t)W^{(0)}) \cdot (W_j^{(1)} - W_j^{(0)}) \, dt.
\end{equation}
\end{proposition}

\subsection{Some Algebraic Formulas}
Suppose $\omega = [\omega_{ij}]_{n \times n} \in S^n,$  where $S^n$ denotes the set of $n \times n$ real symmetric matrices. For any index set $(i_1, i_2, \dots, i_k)$ with $1 \leq i_1 < i_2 < \dots < i_k \leq n,$  we denote the $k \times k$ submatrix:
\begin{equation}\label{eq2.2}
\omega_{i_1 \dots i_k} = [\omega_{i_p i_q}]_{p, q=1}^k.
\end{equation}
Furthermore, we denote
\begin{equation}\label{eq2.3_sym}
\sigma_{j}^{(i_1, \dots, i_d)}(\lambda) = \sigma_{j}(\lambda_1, \dots, \lambda_n)\Big|_{\lambda_{i_1} = \dots = \lambda_{i_d} = 0}, \quad 0\leq d \leq j.
\end{equation}

\begin{proposition}\label{prop2.2}
The following assertions hold:
\begin{enumerate}
\item For any $\omega \in S^n$ with eigenvalues $\lambda(\omega) = (\lambda_1, \dots, \lambda_n),$  we have:
\begin{equation}\label{eq2.2_sk}
S_k(\omega) = \sum_{1 \leq i_1 < \dots < i_k \leq n} \det(\omega_{i_1 \dots i_k}).
\end{equation}
\item Suppose $\omega, \alpha \in S^n.$  Let $\omega = P^T D P$ where $P^T = P^{-1}$ and $D = \operatorname{diag}(\lambda_1, \dots, \lambda_n)$ is the diagonal matrix of eigenvalues of $\omega.$  Let $\tilde{\alpha} = P \alpha P^{-1}.$  Then:
\begin{equation}\label{eq2.3}
S_k(\omega + \alpha) = S_k(\omega) + A_k(\omega, \alpha),
\end{equation}
where
\begin{equation}\label{eq2.4}
A_k(\omega, \alpha) = \sum_{m=1}^k \sum_{1 \leq i_1 < \dots < i_m \leq n} \sigma_{k-m}^{(i_1, \dots, i_m)}(\lambda(\omega)) \det(\tilde{\alpha}_{i_1 \dots i_m}).
\end{equation}
\item Suppose $\omega, \alpha \in S^n.$  Let $\alpha = P^{-1} D P$ where $P^T = P^{-1}$ and $D = \operatorname{diag}(\lambda_1, \dots, \lambda_n)$ is the diagonal matrix of eigenvalues of $\alpha.$  Let $\tilde{\omega} = P \omega P^{-1}.$  Then:
\begin{equation}\label{eq2.5}
S_k(\omega + \alpha) = S_k(\omega) + B_k(\omega, \alpha),
\end{equation}
where
\begin{equation}\label{eq2.6}
B_k(\omega, \alpha) = \sum_{m=1}^k \sum_{1 \leq i_1 < \dots < i_{k-m} \leq n} \det(\tilde{\omega}_{i_1 \dots i_{k-m}}) \sigma_{k-m}^{(i_1, \dots, i_m)}(\lambda(\alpha)).
\end{equation}
\end{enumerate}
\end{proposition}

In particular, we have the following useful relations:
\begin{equation}\label{eq2.3_prime}
\sum_{1 \leq i_1 < \dots < i_m \leq n} \sigma_{k-m}^{(i_1, \dots, i_m)}(\lambda(\omega)) = \frac{(n-k+m)!}{m!(n-k)!} \sigma_{k-m}(\lambda(\omega)), \quad \text{for } m \leq k,
\end{equation}
and for any symmetric matrix $\omega \in S^n:$ 
\begin{equation}\label{eq2.5_prime_prime}
\left| \det(\omega_{i_1 \dots i_m}) \right| \leq \binom{n}{m}^{1/2} \frac{1}{n^{m/2}} |\omega|^m, \quad 0 \leq m \leq n.
\end{equation}
This proposition has been proved in \cite{2}.

\subsection{Proposition on the Sum of Two Matrices}
The following proposition plays a crucial role in establishing the necessary and sufficient conditions for the existence of admissible subsolutions.

\begin{proposition}\label{prop2.3}
Suppose $\omega, \alpha \in S^n$ and $\lambda(\omega) \in \Gamma_k.$  Then the following assertions hold:
\begin{enumerate}
\item If $\alpha \geq 0$ (i.e., $\alpha$ is positive semi-definite), then $\lambda(\omega + \alpha) \in \Gamma_k$ and
\begin{equation}\label{eq2.7}
S_k(\omega + \alpha) \geq S_k(\omega) > 0.
\end{equation}
\item If $\omega \geq 0$ and $\lambda(\alpha) \in \Gamma_k,$  then $\lambda(\omega + c\alpha) \in \Gamma_k$ and
\begin{equation}\label{eq2.8}
S_k(\omega + c\alpha) \geq S_k(\alpha) > 0,\quad c\geq 0.
\end{equation}
\item Suppose $\alpha$ satisfies the condition:
\begin{equation}\label{eq2.9}
|\alpha| \leq \frac{S_1(\omega)}{n} t_k(\omega),
\end{equation}
where $t_k(\omega)$ is the (unique) positive solution of the equation
\begin{equation}\label{eq2.19}
t = g_k(t), \quad t > 0,
\end{equation}
with $g_k(t)$ defined by
\begin{equation}\label{eq2.18}
g_k(t) = \frac{n^k S_k(\omega)}{2 \binom{n}{k} (S_1(\omega))^k \sum_{m=1}^k \binom{n}{m}^{1/2} \frac{1}{n^{m/2}} \binom{k}{m} t^{m-1}}.
\end{equation}
Then $\lambda(\omega + \alpha) \in \Gamma_k$ and
\begin{equation}\label{eq2.10}
S_k(\omega + \alpha) \geq \frac{1}{2} S_k(\omega) > 0.
\end{equation}
\item Suppose $\omega$ satisfies the additional condition: there exist constants $C_3 > 0, C_4 > 0$ such that
\begin{equation}\label{eq2.11}
S_{k-m}(\omega) \leq C_3 C_4^{-\chi} S_k(\omega), \quad 1 \leq m \leq k.
\end{equation}
Suppose $\alpha$ satisfies the condition
\begin{equation}\label{eq2.12}
|\alpha| \leq s_k C_4,
\end{equation}
where $s_k$ is the (unique) positive solution of the equation
\begin{equation}\label{eq2.25}
s = h_k(s), \quad s > 0,
\end{equation}
with $h_k(s)$ defined by
\begin{equation}\label{eq2.24}
h_k(s) = \frac{(n-k)!}{2 C_3 \sum_{m=1}^k \binom{n}{m}^{1/2} \frac{1}{n^{m/2}} \frac{(n-k+m)!}{m!} s^{m-1}}.
\end{equation}
Then $\lambda(\omega + \alpha) \in \Gamma_k$ and
\begin{equation}\label{eq2.13}
S_k(\omega + \alpha) \geq \frac{1}{2} S_k(\omega) > 0.
\end{equation}
\end{enumerate}
\end{proposition}

\begin{proof}
\begin{enumerate}
\item When $\lambda(\omega) \in \Gamma_k,$  we have $\sigma_{k-m}^{(i_1, \dots, i_m)}(\lambda(\omega)) > 0$ for $0 \leq m \leq k$ (see Lin and Trudinger \cite{15}). We have $|\alpha| = |\tilde{\alpha}|.$  Since $\alpha \geq 0,$  we have $\tilde{\alpha} \geq 0,$  and therefore $\det(\tilde{\alpha}_{i_1 \dots i_m}) \geq 0$ for any $1 \leq m \leq k.$  Thus, $A_k(\omega, \alpha) \geq 0,$  where $A_k(\omega, \alpha)$ is defined by \eqref{eq2.4}. Then from \eqref{eq2.3} we obtain \eqref{eq2.7}. Since $\lambda(\omega) \in \Gamma_k,$  it follows by connectedness and the properties of the cone that $\lambda(\omega + \alpha) \in \Gamma_k.$ 

\item If $\omega \geq 0,$  we have $\tilde{\omega} \geq 0$ and $S_k(\omega) = S_k(\tilde{\omega}) \geq 0.$  Moreover, we have $\det(\tilde{\omega}_{i_1 \dots i_m}) \geq 0$ for any submatrix. On the other hand, since $\lambda(\alpha) \in \Gamma_k,$  we have $\sigma_{k-m}^{(i_1, \dots, i_m)}(\lambda(\alpha)) \geq 0.$  From this we obtain $B_k(\omega, \alpha) \geq 0,$  where $B_k(\omega, \alpha)$ is defined by \eqref{eq2.6}. Thus from \eqref{eq2.5} we obtain \eqref{eq2.8}, and since $\lambda(\alpha) \in \Gamma_k,$  we have $\lambda(\omega + \alpha) \in \Gamma_k.$ 

\item By using \eqref{eq2.3_prime}, \eqref{eq2.5_prime_prime}, and the Newton-MacLaurin inequality:
\begin{equation}\label{eq2.newton}
\left[ \frac{\sigma_{k-m}(\lambda(\omega))}{\binom{n}{k-m}} \right]^{\frac{1}{k-m}} \leq \frac{\sigma_1(\lambda(\omega))}{n}, \quad 1 \leq m \leq k,
\end{equation}
we obtain:
\begin{equation}\label{eq2.14}
\begin{aligned}
|A_k(\omega, \alpha)| &\leq \sum_{m=1}^k \sum_{1 \leq i_1 < \dots < i_m \leq n} \sigma_{k-m}^{(i_1, \dots, i_m)}(\lambda(\omega)) \left| \det(\tilde{\alpha}_{i_1 \dots i_m}) \right| \\
&\leq \sum_{m=1}^k \frac{(n-k+m)!}{m!(n-k)!} \sigma_{k-m}(\lambda(\omega)) \binom{n}{m}^{1/2} \frac{1}{n^{m/2}} |\alpha|^m \\
&\leq \sum_{m=1}^k \binom{n}{k} \binom{k}{m} \left( \frac{S_1(\omega)}{n} \right)^{k-m} \binom{n}{m}^{1/2} \frac{1}{n^{m/2}} |\alpha|^m \\
&= \binom{n}{k} \left( \frac{S_1(\omega)}{n} \right)^k \sum_{m=1}^k \binom{n}{m}^{1/2} \frac{1}{n^{m/2}} \binom{k}{m} \left( \frac{|\alpha| n}{S_1(\omega)} \right)^m.
\end{aligned}
\end{equation}
Let $t = \frac{|\alpha| n}{S_1(\omega)}.$  From \eqref{eq2.14} we obtain:
\begin{equation}\label{eq2.16}
|A_k(\omega, \alpha)| \leq t \binom{n}{k} \left( \frac{S_1(\omega)}{n} \right)^k \sum_{m=1}^k \binom{n}{m}^{1/2} \frac{1}{n^{m/2}} \binom{k}{m} t^{m-1}.
\end{equation}
We determine $t > 0$ such that:
\begin{equation}\label{eq2.17}
t \binom{n}{k} \left( \frac{S_1(\omega)}{n} \right)^k \sum_{m=1}^k \binom{n}{m}^{1/2} \frac{1}{n^{m/2}} \binom{k}{m} t^{m-1} \leq \frac{1}{2} S_k(\omega),
\end{equation}
which is equivalent to:
\begin{equation}\label{eq2.18_ineq}
t \leq \frac{n^k S_k(\omega)}{2 \binom{n}{k} (S_1(\omega))^k \sum_{m=1}^k \binom{n}{m}^{1/2} \frac{1}{n^{m/2}} \binom{k}{m} t^{m-1}} =: g_k(t).
\end{equation}
Since $g_k(t)$ is positive and monotonically decreasing for $t > 0,$  the equation $t = g_k(t)$ has a unique positive solution $t_k(\omega).$  The solution to the inequality \eqref{eq2.18_ineq} is then given by $0 < t \leq t_k(\omega).$ 
Therefore, if $\alpha$ satisfies $|\alpha| \leq \frac{S_1(\omega)}{n} t_k(\omega),$  we have $|A_k(\omega, \alpha)| \leq \frac{1}{2} S_k(\omega).$ 
Thus,
\begin{equation}
S_k(\omega + \alpha) \geq S_k(\omega) - |A_k(\omega, \alpha)| \geq \frac{1}{2} S_k(\omega) > 0
\end{equation}

\item Since $\sigma_{k-m}(\lambda(\omega)) = S_{k-m}(\omega),$  from \eqref{eq2.20} and the hypothesis \eqref{eq2.11}, we obtain:
\begin{equation}\label{eq2.20}
\begin{aligned}
|A_k(\omega, \alpha)| &\leq \sum_{m=1}^k \binom{n}{m}^{1/2} \frac{1}{n^{m/2}} \frac{(n-k+m)!}{m! (n-k)!} S_{k-m}(\omega) |\alpha|^m \\
&\leq \sum_{m=1}^k \binom{n}{m}^{1/2} \frac{1}{n^{m/2}} \frac{(n-k+m)!}{m! (n-k)!} C_3 C_4^{-\chi} S_k(\omega) |\alpha|^m \\
&= C_3 S_k(\omega) \sum_{m=1}^k \binom{n}{m}^{1/2} \frac{1}{n^{m/2}} \frac{(n-k+m)!}{m! (n-k)!} \left( \frac{|\alpha|}{C_4} \right)^m.
\end{aligned}
\end{equation}
Let $s = \frac{|\alpha|}{C_4}.$  From \eqref{eq2.20} we have:
\begin{equation}\label{eq2.22}
|A_k(\omega, \alpha)| \leq s \frac{C_3 S_k(\omega)}{(n-k)!} \sum_{m=1}^k \binom{n}{m}^{1/2} \frac{1}{n^{m/2}} \frac{(n-k+m)!}{m!} s^{m-1}.
\end{equation}
We determine $s > 0$ such that:
\begin{equation}\label{eq2.23}
s \frac{C_3 S_k(\omega)}{(n-k)!} \sum_{m=1}^k \binom{n}{m}^{1/2} \frac{1}{n^{m/2}} \frac{(n-k+m)!}{m!} s^{m-1} \leq \frac{1}{2} S_k(\omega),
\end{equation}
which is equivalent to:
\begin{equation}\label{eq2.24_ineq}
s \leq \frac{(n-k)!}{2 C_3 \sum_{m=1}^k \binom{n}{m}^{1/2} \frac{1}{n^{m/2}} \frac{(n-k+m)!}{m!} s^{m-1}} =: h_k(s).
\end{equation}
We denote by $s_k$ the unique positive solution of the equation $s = h_k(s).$  Since $h_k(s)$ is positive and monotonically decreasing for $s > 0,$  this solution is unique.
Then, if $s \leq s_k$ (which is equivalent to $|\alpha| \leq s_k C_4$), we have $|A_k(\omega, \alpha)| \leq \frac{1}{2} S_k(\omega).$ 
Consequently,
\begin{equation}
S_k(\omega + \alpha) \geq S_k(\omega) - |A_k(\omega, \alpha)| \geq \frac{1}{2} S_k(\omega) > 0.
\end{equation}
The proof is complete.
\end{enumerate}
\end{proof}

\section{Proof of Theorem 1.1}

Since $v(x) \in C^2(\bar{\Omega}_{\delta})$ is an admissible function of equation \eqref{eq1.1}, we have for the matrix:
\begin{equation}\label{eq3.1}
\omega_1(v)(x):= D^2 v(x) - A(x, v(x), Dv(x)),
\end{equation}
the property:
\begin{equation}\label{eq3.2_ad}
S_k(\omega_1(v)(x)) > 0 \quad \text{for any } x \in \bar{\Omega}_{\delta},
\end{equation}
which implies in particular that $\lambda(\omega_1(v)(x_0)) \in \Gamma_k$ for any $x_0 \in \partial\Omega.$ 

Next, we define:
\begin{equation}\label{eq3.2}
\begin{aligned}
\omega_2(x_0) &:= D^2 v(x_0) - \sum_{j=1}^n D_{p_j} A(x_0, 0, 0) D_{x_j} v(x_0) \\
&= \left( D^2 v(x_0) - A(x_0, 0, Dv(x_0)) \right) + A(x_0, 0, 0) \\
&\quad + \left[ A(x_0, 0, Dv(x_0)) - A(x_0, 0, 0) - \sum_{j=1}^n D_{p_j} A(x_0, 0, 0) D_{x_j} v(x_0) \right] \\
&=: \omega_1(v)(x_0) + \alpha_1(x_0) + \alpha_2(x_0).
\end{aligned}
\end{equation}

Applying the Lagrange formula \eqref{eq2.1} of Proposition \ref{prop2.1}, we have:
\begin{equation}\label{eq3.lag}
A(x_0, 0, Dv(x_0)) - A(x_0, 0, 0) = \sum_{j=1}^n \int_0^1 D_{p_j} A(x_0, 0, t Dv(x_0)) D_{x_j} v(x_0) \, dt.
\end{equation}
Therefore,
\begin{equation}\label{eq3.alpha2}
\begin{aligned}
\alpha_2(x_0) &= \sum_{j=1}^n \int_0^1 \left( D_{p_j} A(x_0, 0, t Dv(x_0)) - D_{p_j} A(x_0, 0, 0) \right) D_{x_j} v(x_0) \, dt \\
&= \sum_{i,j=1}^n \int_0^1 \int_0^1 D^2_{p_i p_j} A(x_0, 0, t s Dv(x_0)) D_{x_i} v(x_0) D_{x_j} v(x_0) \, dt \, ds.
\end{aligned}
\end{equation}
From this, we obtain the estimate:
\begin{equation}\label{eq3.est_alpha2}
|\alpha_2(x_0)| \leq |Dv(x_0)|^2 \sup_{|p| \leq |Dv(x_0)|} \left[ \sum_{i,j=1}^n |D^2_{p_i p_j} A(x_0, 0, p)|^2 \right]^{1/2}.
\end{equation}

From \eqref{eq3.est_alpha2} and condition (4) of Theorem \ref{thm1.1}, we have:
\begin{equation}\label{eq3.bound_alpha2}
|\alpha_2(x_0)| \leq \frac{1}{2n} S_k(\omega_1(v)(x_0)) \cdot t_k(\omega_1(v)(x_0)),
\end{equation}
where $t_k(\omega_1(v)(x_0))$ is the unique positive solution of the equation $t = g_k(t)$ (where $g_k(t)$ is defined in \eqref{eq2.18} with $\omega = \omega_1(v)(x_0)$).
By part (3) of Proposition \ref{prop2.3}, we have $\lambda(\omega_1(v)(x_0) + \alpha_2(x_0)) \in \Gamma_k,$  and
\begin{equation}\label{eq3.4}
S_k(\omega_1(v)(x_0) + \alpha_2(x_0)) > 0, \quad x_0 \in \partial\Omega.
\end{equation}
Furthermore, since $\alpha_1(x_0) = A(x_0, 0, 0) \geq 0$ by condition (3) of Theorem \ref{thm1.1}, part (1) of Proposition \ref{prop2.3} yields:
\begin{equation}\label{eq3.5}
S_k(\omega_2(x_0)) = S_k( \omega_1(v)(x_0) + \alpha_2(x_0) + \alpha_1(x_0) ) \geq S_k(\omega_1(v)(x_0) + \alpha_2(x_0)) > 0, \quad x_0 \in \partial\Omega.
\end{equation}

Now, we consider the principal coordinate system $O' e'_1 \dots e'_n,$  where $O' = x_0$ and $e'_n = -\nu(x_0).$  We perform the following change of variables between the two coordinate systems:
\begin{equation}\label{eq3.coord}
x - x_0 = J(x_0) y, \quad (J(x_0))^T = (J(x_0))^{-1},
\end{equation}
where $J(x_0) = [J_{ij}(x_0)]_{n \times n}$ is the transition matrix from the old basis $e_1, \dots, e_n$ to the new basis $e'_1, \dots, e'_n.$ 
We define:
\begin{equation}\label{eq3.v_tilde}
v(x) = v(x_0 + J(x_0) y) =: \tilde{v}(x_0, y).
\end{equation}
Then,
\begin{equation}\label{eq3.D_tilde}
D_x v(x) = J(x_0) D_y \tilde{v}(x_0, y),
\end{equation}
\begin{equation}\label{eq3.D2_tilde}
D^2_x v(x) = J(x_0) D^2_y \tilde{v}(x_0, y) (J(x_0))^T.
\end{equation}
When $x = x_0 \in \partial\Omega$ (which corresponds to $y = 0$), we have:
\begin{equation}\label{eq3.D_x0}
D_x v(x_0) = J(x_0) D_y \tilde{v}(x_0, 0),
\end{equation}
\begin{equation}\label{eq3.D2_x0}
D^2_x v(x_0) = J(x_0) D^2_y \tilde{v}(x_0, 0) (J(x_0))^T.
\end{equation}

In the principal coordinate system, the boundary $\partial\Omega$ is locally represented by:
\begin{equation}\label{eq3.boundary}
y_n = \rho(y'), \quad y' = (y_1, \dots, y_{n-1}),
\end{equation}
where $\rho(0') = 0$ and $D_{y'} \rho(0') = 0.$  Since $v(x) = 0$ on $\partial\Omega$ by condition 1) of Theorem \ref{thm1.1}, we have the identity:
\begin{equation}\label{eq3.identity}
\tilde{v}(x_0, (y', \rho(y'))) = 0.
\end{equation}
Differentiating \eqref{eq3.identity} with respect to $y_i$ ($1 \leq i \leq n-1$), we obtain:
\begin{equation}\label{eq3.diff1}
\tilde{v}_{y_i} + \tilde{v}_{y_n} \rho_{y_i} = 0,
\end{equation}
and differentiating once more with respect to $y_j$ ($1 \leq j \leq n-1$), we get:
\begin{equation}\label{eq3.diff2}
\tilde{v}_{y_i y_j} + \tilde{v}_{y_i y_n} \rho_{y_j} + \left( \tilde{v}_{y_n y_i} + \tilde{v}_{y_n y_n} \rho_{y_i} \right) \rho_{y_j} + \tilde{v}_{y_n} \rho_{y_i y_j} = 0.
\end{equation}
At $y = 0,$  since $D_{y'} \rho(0') = 0,$  we have:
\begin{equation}\label{eq3.v_y_i}
\tilde{v}_{y_i}(x_0, 0) = 0, \quad i = 1, \dots, n-1,
\end{equation}
\begin{equation}\label{eq3.v_y_ij}
\tilde{v}_{y_i y_j}(x_0, 0) = -\tilde{v}_{y_n}(x_0, 0) \rho_{y_i y_j}(x_0, 0'), \quad 1 \leq i, j \leq n-1,
\end{equation}
where $\tilde{v}_{y_n}(x_0, 0) = -|Dv(x_0)| < 0$ (due to $-\nu(x_0)$ being the direction of the positive $y_n$-axis, and the fact that $v(x) < 0$ in $\Omega$ near $\partial\Omega$).

Since $Dv(x_0)=|Dv(x_0)|\nu(x_0),$ from \eqref{eq3.D2_x0}, we can express the matrix $\omega_2(x_0)$ as:
\begin{equation}\label{eq3.omega2_coord}
\begin{aligned}
\omega_2(x_0) &= D^2 v(x_0) - \sum_{j=1}^n D_{p_j} A(x_0, 0, 0) D_{x_j} v(x_0) \\
&= J(x_0) \left[ D^2_y \tilde{v}(x_0, 0) - |Dv(x_0)| (J(x_0))^T \sum_{l=1}^n D_{p_l} A(x_0, 0, 0) \nu_l(x_0) J(x_0) \right] (J(x_0))^T \\
&= J(x_0) \tilde{\omega}_2(x_0, 0) (J(x_0))^T,
\end{aligned}
\end{equation}
where
\begin{equation}\label{eq3.omega2_tilde}
\begin{aligned}
\tilde{\omega}_2(x_0, 0) &:= |Dv(x_0)| \left[ D_y^2\tilde{v}(x_0,0)- (J(x_0))^T \sum_{j=1}^n D_{p_j} A(x_0, 0, 0) \nu_j(x_0) J(x_0) \right] \\
&= \left[ \tilde{\omega}_{ij}(x_0, 0) \right]_{n \times n} \\
&= |Dv(x_0)| \begin{bmatrix} M(x_0, A) & * \\ * & \frac{\tilde{\omega}_{nn}(x_0, 0)}{|Dv(x_0)|} \end{bmatrix},
\end{aligned}
\end{equation}
with $M(x_0, A)$ being the $(n-1) \times (n-1)$ matrix defined in \eqref{eq1.4}, do not depend on $\tilde{\omega}_{nn}.$

Since $\lambda(\omega_2(x_0)) \in \Gamma_k,$  and $\tilde{\omega}_2(x_0, 0)$ is orthogonally similar to $\omega_2(x_0)$ (as $J(x_0)$ is orthogonal), we have $\lambda(\tilde{\omega}_2(x_0, 0)) \in \Gamma_k.$  It is a well-known property of the elementary symmetric polynomials on the cone $\Gamma_k$ that:
\begin{equation}\label{eq3.pos_derivatives}
\frac{\partial S_k(\tilde{\omega}_2(x_0, 0))}{\partial \tilde{\omega}_{ij}} \Big|_{n \times n} > 0 \quad \text{and in particular} \quad \frac{\partial S_k(\tilde{\omega}_2(x_0, 0))}{\partial \tilde{\omega}_{nn}} > 0.
\end{equation}
From \eqref{eq3.omega2_tilde} and \eqref{eq2.2_sk}, we can expand $S_k(\tilde{\omega}_2(x_0, 0))$ with respect to its last diagonal entry $\tilde{\omega}_{nn}(x_0, 0):$ 
\begin{equation}\label{eq3.expansion}
S_k(\tilde{\omega}_2(x_0, 0)) = |Dv(x_0)|^k \left[ \frac{\tilde{\omega}_{nn}(x_0, 0)}{|Dv(x_0)|} S_{k-1}(M(x_0, A)) + S_k(M(x_0, A)) + P_k \right],
\end{equation}
where $P_k$ is a term independent of $\tilde{\omega}_{nn}(x_0, 0).$ 
Differentiating \eqref{eq3.expansion} with respect to $\tilde{\omega}_{nn}(x_0, 0),$  we obtain:
\begin{equation}\label{eq3.diff_nn}
\frac{\partial S_k(\tilde{\omega}_2(x_0, 0))}{\partial \tilde{\omega}_{nn}(x_0, 0)} = |Dv(x_0)|^{k-1} S_{k-1}(M(x_0, A)).
\end{equation}
From \eqref{eq3.pos_derivatives} and \eqref{eq3.diff_nn}, since $|Dv(x_0)| > 0$ by condition (2) of Theorem \ref{thm1.1}, we must have:
\begin{equation}\label{eq3.pos_sk_minus_1}
S_{k-1}(M(x_0, A)) > 0 \quad \text{for any } x_0 \in \partial\Omega.
\end{equation}
Since $\partial\Omega$ is compact, and the matrix function $M(x_0, A)$ depends continuously on $x_0$ for $x_0 \in \partial\Omega,$  we obtain:
\begin{equation}\label{eq3.inf_pos}
\inf_{x_0 \in \partial\Omega} S_{k-1}(M(x_0, A)) > 0,
\end{equation}
which implies that the domain $\Omega$ is uniformly $(k-1)$-$A$-convex. The proof of Theorem \ref{thm1.1} is complete. $\square$

\section{Proof of Theorem \ref{thm1.2}: Construction of Admissible Subsolutions}

Theorem \ref{thm1.2} will be proved through several steps, in which we apply and adapt ideas and techniques to construct an admissible function and an admissible subsolution explicitly in closed form. To apply assertion (4) of Proposition \ref{prop2.3}, our main tool is a series of precise norm estimates on several auxiliary matrices in a boundary neighborhood $\Omega_\delta.$ 

\subsection{Choice of Parameters}

We choose five fixed positive parameters $\gamma_j \in (0, 1)$ ($j=1, \dots, 5$) satisfying the structural relations:
\begin{equation}\label{eq4.1}
0 < \gamma_2 < \gamma_5 < \frac{1}{2} \gamma_1 + \gamma_2 < \gamma_3 = \gamma_1 + \gamma_2 < 2\gamma_1 + \gamma_2 < \gamma_4< 1 = 3\gamma_1 + \gamma_2.
\end{equation}
For a large parameter $a > 1,$  we set:
\begin{equation}\label{eq4.2}
t_a = a^{\gamma_5}, \quad c_a = \frac{1}{a^{\gamma_4}}, \quad \alpha_1(a) = 1 - \frac{1}{2 a^{\gamma_3 - \gamma_5}}, \quad \theta_a(x) = \frac{d(x)}{a^{\gamma_2}}, \quad x \in \bar{\Omega},
\end{equation}
where $d(x) = d(x, \partial\Omega) = \inf_{x_0 \in \partial\Omega} |x - x_0|$ denotes the distance from $x$ to the boundary $\partial\Omega.$ 
For $\delta > 0,$  we denote the boundary neighborhood by:
\begin{equation}\label{eq4.3}
\Omega_\delta = \{ x \in \Omega: d(x) < \delta \}.
\end{equation}
The relationship between $a > 0$ and $\delta > 0$ is specified by:
\begin{equation}\label{eq4.4}
\delta = \frac{1}{a^{\gamma_1}}.
\end{equation}

We choose $a_1 > 0$ sufficiently large such that:
\begin{equation}\label{eq4.5}
\frac{1}{a_1^{\gamma_1}} < \sup_{x_0 \in \partial\Omega} \max_{1 \leq j \leq n-1} |\kappa_j(x_0)|,
\end{equation}
where $\kappa_1(x_0), \dots, \kappa_{n-1}(x_0)$ are the principal curvatures of $\partial\Omega$ at $x_0 \in \partial\Omega$ with respect to the inner unit normal vector $-\nu(x_0).$  When $a \geq a_1,$  the distance function $d(x) \in C^2(\bar{\Omega}_{1/a^{\gamma_1}}).$  Moreover, in the orthonormal principal coordinate system $O' e'_1 \dots e'_n$ centered at $O' = x_0,$  we have:
\begin{equation}\label{eq4.6}
D_y d(y) = (0, \dots, 0, 1) = D_x d(x) J(x_0),
\end{equation}
\begin{equation}\label{eq4.7}
D^2_y d(y) = \operatorname{diag}\left[ \frac{-\kappa_1(x_0)}{1 - \kappa_1(x_0) d(y)}, \dots, \frac{-\kappa_{n-1}(x_0)}{1 - \kappa_{n-1}(x_0) d(y)}, 0 \right] = J(x_0)^T D^2_x d(x) J(x_0),
\end{equation}
where $J(x_0) = [J_{ij}(x_0)]_{n \times n}$ is the orthogonal transformation matrix ($J(x_0)^T = J(x_0)^{-1}$) from the standard basis $e_1, \dots, e_n$ to the local principal basis $e'_1, \dots, e'_{n-1}, e'_n = -\nu(x_0),$  with local coordinates $x = x_0 + d(x) y J(x_0).$ 
From \eqref{eq4.6}, for any $a \geq a_1$ and $x \in \Omega_{1/a^{\gamma_1}},$  we have:
\begin{equation}\label{eq4.9}
D_x d(x) = D_x d(x_0) = -\nu(x_0),
\end{equation}
where $x_0 \in \partial\Omega$ is the unique boundary point satisfying $|x - x_0| = d(x).$ 

\subsection{Construction of the Auxiliary Function $v_a(x)$}

We introduce the auxiliary function $v_a(x)$ defined in $\bar{\Omega}_{1/a^{\gamma_1}}$ by:
\begin{equation}\label{eq4.10}
v_a(x) = \frac{e^{-t_a \theta_a(x)} - 1}{t_a}, \quad x \in \bar{\Omega}_{1/a^{\gamma_1}}.
\end{equation}
Differentiating \eqref{eq4.10} yields:
\begin{equation}\label{eq4.11}
D v_a(x) = - e^{-t_a \theta_a(x)} D \theta_a(x),
\end{equation}
\begin{equation}\label{eq4.12}
D^2 v_a(x) = t_a e^{-t_a \theta_a(x)} D\theta_a(x) \otimes D\theta_a(x) - e^{-t_a \theta_a(x)} D^2 \theta_a(x).
\end{equation}
For any boundary point $x_0 \in \partial\Omega,$  we obtain:
\begin{equation}\label{eq4.13}
D \theta_a(x_0) = \frac{1}{a^{\gamma_2}} \nu(x_0), \quad D v_a(x_0) = \frac{-1}{a^{\gamma_2}} \nu(x_0),
\end{equation}
\begin{equation}\label{eq4.14}
D^2 v_a(x_0) = t_a D\theta_a(x_0) \otimes D\theta_a(x_0) - D^2 \theta_a(x_0).
\end{equation}
Furthermore, $v_a(x)$ depends only on the distance $d(x),$  so we can write:
\begin{equation}\label{eq4.15}
v_a(x) = \bar{v}_a(d(x)), \quad \text{where } \bar{v}_a(d) = \frac{e^{-t_a d / a^{\gamma_2}} - 1}{t_a},\quad 0\leq d\leq \frac{1}{a^{\gamma_1}}.
\end{equation}
Recall that for $0\leq \eta <1,$  the exponential function satisfies:
\begin{equation}\label{eq4.16}
-\eta \leq e^{-\eta} - 1 \leq -\eta + \frac{1}{2}\eta^2 = -\eta \left(1 - \frac{1}{2}\eta\right).
\end{equation}
Applying \eqref{eq4.16} to \eqref{eq4.15}, we obtain for $0<d\leq \frac{1}{a^{\gamma_1}}:$ 
\begin{equation}\label{eq4.16'}
\frac{d}{a^{\gamma_2}}\alpha_1(a)=\frac{d}{a^{\gamma_2}} \big(1-\frac{1}{2a^{\gamma_3-\gamma_5}} \big)\leq  \frac{d}{a^{\gamma_2}} \left(1 - \frac{t_a d}{2 a^{\gamma_2}}\right) < -\bar{v}_a(d) < \frac{d}{a^{\gamma_2}}.
\end{equation}
where $\alpha_1(a)$ is defined by \eqref{eq4.2}. 

We choose $a_2 > a_1$ sufficiently large such that:
\begin{equation}\label{eq4.16''}
\alpha_1(a_2) = 1 - \frac{1}{2 a_2^{\gamma_3 - \gamma_5}} > \frac{2}{\sqrt{5}}.
\end{equation}
Thus, for all $a \geq a_2,$  we have:
\begin{equation}\label{eq4.17}
1>\alpha_1(a) \geq \alpha_1(a_2) > \frac{2}{\sqrt{5}}, \quad \alpha_1(a) > \frac{4}{5} (\alpha_1(a))^{-1}.
\end{equation}
Combining \eqref{eq4.2}, \eqref{eq4.15}, \eqref{eq4.16} and \eqref{eq4.16'}, we obtain for $a \geq a_2$ and $d > 0:$ 
\begin{equation}\label{eq4.18}
\frac{d}{a^{\gamma_2}} \alpha_1(a) < -\bar{v}_a(d) < \frac{d}{a^{\gamma_2}}.
\end{equation}
We set $\delta_1(a) = \frac{1}{a^{\gamma_1}}$
  and define $s_a$ by:
\begin{equation}\label{eq4.20}
s_a = -\bar{v}_a(\delta_1) = \frac{1 - e^{-1/a^{\gamma_3 - \gamma_5}}}{a^{\gamma_5}}.
\end{equation}
From \eqref{eq4.18}, it follows that:
\begin{equation}\label{eq4.21}
\frac{1}{a^{\gamma_3}} \alpha_1(a) < s_a < \frac{1}{a^{\gamma_3}}.
\end{equation}

\subsection{Choice of Additional Parameters $\delta_j(a)$ and Construction of $g_s(r)$}
From \eqref{eq4.16'} we have 
\begin{equation}\label{eq4.22}
	a^{\gamma_2} (-\bar{v}_a(d)) < d < a^{\gamma_2} (-\bar{v}_a(d)) (\alpha_1(a))^{-1},\quad d>0.
\end{equation}
Since $\bar{v}_a(d)$ is strictly decreasing on $[0, \delta_1],$  we can uniquely define $\delta_2(a), \delta_3(a), \delta_4(a) > 0$ such that the following relations hold:
\begin{equation}\label{eq4.23}
\bar{v}_a(\delta_2(a)) = -\frac{1}{6} s_a,
\end{equation}
\begin{equation}\label{eq4.24}
\bar{v}_a(\delta_3(a)) = -\frac{1}{2} s_a,
\end{equation}
\begin{equation}\label{eq4.25}
\bar{v}_a(\delta_4(a)) = -\frac{5}{6} s_a.
\end{equation}
We define also
\begin{equation}\label{eq4.25'}
\delta_5(a):=a^{-\gamma_1}\big(\frac{5}{6}\alpha_1(a)-\frac 2 3 (\alpha_1(a))^{-1}\big),
\end{equation}
\begin{equation}\label{eq4.25''}
\delta_6(a):=a^{-\gamma_1}\big((\alpha_1(a))^{-1}-\frac 1 2 \alpha_1(a)\big).
\end{equation}
By \eqref{eq4.22}--\eqref{eq4.25} and  \eqref{eq4.18}, we obtain the bounds:
\begin{equation}\label{eq4.26}
\frac{1}{6} a^{-\gamma_1} \alpha_1(a) < \delta_2(a) < \frac{1}{6} a^{-\gamma_1} (\alpha_1(a))^{-1},
\end{equation}
\begin{equation}\label{eq4.27}
\frac{1}{2} a^{-\gamma_1} \alpha_1(a) < \delta_3(a) < \frac{1}{2} a^{-\gamma_1} (\alpha_1(a))^{-1},
\end{equation}
\begin{equation}\label{eq4.28}
\frac{5}{6} a^{-\gamma_1} \alpha_1(a) < \delta_4(a) < \frac{5}{6} a^{-\gamma_1} (\alpha_1(a))^{-1}.
\end{equation}
From \eqref{eq4.16'} it follows that when $a\geq a_2$ we have $\delta_5(a)>0,\delta_6(a)>0.$

For $s > 0,$  we introduce the polynomial function $g_s(r)$ defined on $[-s, 0]$ by:
\begin{equation}\label{eq4.29}
g_s(r):= \frac{1}{5 s^4} \left[ (s + r)^5 - s^5 \right].
\end{equation}
From \eqref{eq4.29}, $g_s(r)$ satisfies the following estimates on $[-s, 0]:$ 
\begin{equation}\label{eq4.30}
-\frac{1}{5} s \leq g_s(r) \leq g_s(0) = 0,
\end{equation}
\begin{equation}\label{eq4.31}
0 \leq g'_s(r) = \frac{1}{s^4} (s + r)^4 \leq g'_s(0) = 1,
\end{equation}
\begin{equation}\label{eq4.32}
0 \leq g''_s(r) = \frac{4}{s^4} (s + r)^3 \leq g''_s(0) = \frac{4}{s}.
\end{equation}
When $r \leq -s,$  we extend $g_s(r)$ as a constant function $\tilde{g}_s(r) = -\frac{1}{5} s.$  The resulting function $\tilde{g}_s(r) \in C^4((-\infty, 0]).$ 

\subsection{Construction of the Cut-off Function $\zeta_a(x)$}

We construct a smooth cut-off function $\zeta_a(x) \in C^\infty_0(\Omega)$ such that $\zeta_a(x) = 1$ outside a sufficiently small neighborhood of $\partial\Omega.$  For a compact subset $H \subset\subset \Omega$ and $\varepsilon > 0,$  let $H^{\varepsilon} = \{ x \in \Omega: d(x, H) < \varepsilon \}.$  We set:
\begin{equation}\label{eq4.33}
K:= \Omega \setminus \Omega_{\delta_4(a)},
\end{equation}
where $\delta_4(a)$ is defined by \eqref{eq4.25} and satisfies \eqref{eq4.28}. Thus, $d(K, \partial\Omega) = \delta_4(a).$ 
Set $\varepsilon = \delta_2(a)$ and $\varepsilon' = \delta_3(a).$  From \eqref{eq4.26}--\eqref{eq4.28}, we have:
\begin{gather}
\frac{2}{3}\alpha_1(a)\frac{1}{a^{\gamma_1}}<\varepsilon+\varepsilon'<\frac{2}{3}(\alpha_1(a))^{-1}\frac{1}{a^{\gamma_1}},\notag\\
\frac{1}{a^{\gamma_1}}\big(\frac{1}{2}\alpha_1(a)-\frac 1 6(\alpha_1(a))^{-1}\big)<\varepsilon'-\varepsilon<\big(\frac{1}{2}(\alpha_1(a))^{-1}-\frac 1 6\alpha_1(a)\big)\frac{1}{a^{\gamma_1}},\notag\\
\delta_5(a):=\frac{1}{a^{\gamma_1}}\big(\frac{5}{6}\alpha_1(a)-\frac 2 3(\alpha_1(a))^{-1}\big)<\delta_4(a)-(\varepsilon+\varepsilon')<\big(\frac{5}{6}(\alpha_1(a))^{-1}-\frac 2 3\alpha_1(a)\big)\frac{1}{a^{\gamma_1}},\label{eq4.34}\\
\frac{1}{a^{\gamma_1}}\big(\alpha_1(a)-\frac 1 2(\alpha_1(a))^{-1}\big)<\delta_4(a)-(\varepsilon'-\varepsilon)<\big((\alpha_1(a))^{-1}-\frac 1 2\alpha_1(a)\big)\frac{1}{a^{\gamma_1}}=:\delta_6(a).\label{eq4.35}
\end{gather}
For $a\geq a_2,$ from \eqref{eq4.16'} we get $\alpha_1(a)>\frac{2}{\sqrt{5}},$ so that $\delta_6(a)>\delta_5(a)>0.$ Now, by \eqref{eq4.34}, \eqref{eq4.35}, we have  
 $$K^{\varepsilon' + \varepsilon} \subset \Omega \setminus \Omega_{\delta_5(a)} \text{ and } K^{\varepsilon' - \varepsilon} \supset \Omega \setminus \Omega_{\delta_6(a)}.$$ 

Let $\psi(x) \in C^\infty_0(\mathbb{R}^n)$ be a standard mollifying kernel with $0 \leq \psi(x) \leq 1,$  $\operatorname{supp} \psi = \bar{B}_1(0),$  $\int_{B_1(0)} \psi(x) dx = 1,$  $\int_{B_1(0)} |D\psi(x)| dx \leq C_5,$  and $\int_{B_1(0)} |D^2\psi(x)| dx \leq C_6.$ 
With $\varepsilon = \delta_2(a),$  we define:
\begin{equation}\label{eq4.38}
\zeta_a(x) = \frac{1}{\varepsilon^n} \int_{K^{\varepsilon'}} \psi\left(\frac{x - y}{\varepsilon}\right) dy.
\end{equation}
The function $\zeta_a(x) \in C^\infty_0(\Omega)$ satisfies:
\begin{equation}\label{eq4.39}
0 \leq \zeta_a(x) \leq 1,
\end{equation}
\begin{equation}\label{eq4.40}
\zeta_a(x) = 0 \quad \text{in } \Omega_{\delta_5(a)},
\end{equation}
\begin{equation}\label{eq4.41}
\zeta_a(x) = 1 \quad \text{in } \Omega \setminus \Omega_{\delta_6(a)},
\end{equation}
\begin{equation}\label{eq4.42}
\operatorname{supp} \zeta_a \subset \Omega \setminus \Omega_{\delta_5(a)},
\end{equation}
\begin{equation}\label{eq4.43}
|D\zeta_a(x)| \leq C_6 a^{\gamma_1}, \quad |D^2\zeta_a(x)| \leq C_7 a^{2\gamma_1}.
\end{equation}

\subsection{Extension of Boundary Values $\phi(x)$ into the Interior}

Let $g_{s_a}(r)$ be defined by \eqref{eq4.29} on $[-s_a, 0],$  and let $v_a(x)$ be defined by \eqref{eq4.10}. We define $h_a(x)$ in $\Omega_{\delta_1(a)}$ by:
\begin{equation}\label{eq4.44}
h_a(x) = -\frac{5}{s_a} g_{s_a}(s_a - v_a(x)), \quad x \in \Omega_{\delta_1(a)}.
\end{equation}
Direct differentiation gives:
\begin{equation}\label{eq4.45}
D h_a(x) = \frac{5}{s_a} g'_{s_a}(s_a - v_a(x)) D v_a(x),
\end{equation}
\begin{align}\label{eq4.46}
D^2 h_a(x) &= \frac{5}{s_a} g''_{s_a}(s_a - v_a(x)) e^{-2 t_a \theta_a(x)} D\theta_a(x) \otimes D\theta_a(x) \nonumber \\
&\quad + \frac{5}{s_a} g'_{s_a}(s_a - v_a(x)) e^{-t_a \theta_a(x)} D\theta_a(x) \otimes D\theta_a(x) \nonumber \\
&\quad - \frac{5}{s_a} g'_{s_a}(s_a - v_a(x)) e^{-t_a \theta_a(x)} D^2 \theta_a(x).
\end{align}
From \eqref{eq4.30}--\eqref{eq4.32}, \eqref{eq4.1}, \eqref{eq4.2}, \eqref{eq4.6}, and \eqref{eq4.7}, we deduce:
\begin{equation}\label{eq4.47}
0 \leq h_a(x) \leq h_a(x_0)=1, \quad x_0 \in \partial\Omega,
\end{equation}
\begin{equation}\label{eq4.48}
|D h_a(x)| \leq \frac{5}{s_a} |D v_a(x)| \leq \delta a^{\gamma_1},
\end{equation}
\begin{equation}\label{eq4.49}
|D^2 h_a(x)| \leq C_8 a^{2\gamma_1},
\end{equation}
where $C_8 > 0$ depends on the maximum principal curvature $\sup_{x_0 \in \partial\Omega} \max_{1 \leq j \leq n-1} |\kappa_j(x_0)|.$ 

Using the cut-off function $\zeta_a(x) \in C^\infty_0(\Omega)$ from \eqref{eq4.38}, we define the extension $\phi^*(x)$ of $\phi(x)$ from $\partial\Omega$ into $\bar{\Omega}$ by:
\begin{equation}\label{eq4.50}
\phi^*(x) = \phi(x_0) (1 - \zeta_a(x)) h_a(x),
\end{equation}
where $x_0 \in \partial\Omega$ is the unique boundary point nearest to $x.$  Since $h_a(x)$ and $\zeta_a(x)$ belong to $C^2(\bar{\Omega}),$  $\phi^*(x) \in C^2(\bar{\Omega})$ and from \eqref{eq4.43}, \eqref{eq4.48}, \eqref{eq4.49} we have :
\begin{equation}\label{eq4.51}
\phi^*(x) = \phi(x), \quad x \in \partial\Omega,
\end{equation}
\begin{equation}\label{eq4.52}
\phi^*(x) = 0, \quad x \in \Omega \setminus \Omega_{\delta_6(a)},
\end{equation}
\begin{equation}\label{eq4.53}
|D\phi^*(x)| \leq C_9 a^{\gamma_1}, \quad x \in \bar{\Omega},
\end{equation}
\begin{equation}\label{eq4.54}
|D^2\phi^*(x)| \leq C_{10} a^{2\gamma_1}, \quad x \in \bar{\Omega}.
\end{equation}

\subsection{Construction of the Candidate Subsolution $w_a(x)$}

We extend $g_{s_a}(v_a(x))$ to the whole domain $\Omega$ by defining:
\begin{equation}\label{eq4.55}
\tilde{g}_{s_a}(v_a(x)) = \begin{cases} g_{s_a}(v_a(x)), & x \in \bar{\Omega}_{\delta_1(a)}, \\ -\frac{1}{5} s_a, & x \in \Omega \setminus \bar{\Omega}_{\delta_1(a)}. \end{cases}
\end{equation}
Since $\tilde{g}_{s_a}(r) \in C^4((-\infty, 0])$ and $v_a(x) \in C^2(\bar{\Omega}_{\delta_1(a)}),$  we have $\tilde{g}_{s_a}(v_a(x)) \in C^2(\bar{\Omega})$ and:
\begin{equation}\label{eq4.56}
\tilde{g}_{s_a}(v_a(x_0)) = 0, \quad x_0 \in \partial\Omega.
\end{equation}

Fix a point $x^* = (x^*_1, \dots, x^*_n) \in \Omega$ satisfying $d(x^*, \partial\Omega) > \delta_1(a_2),$  where $a_2 > 0$ is specified by \eqref{eq4.16''}.
We set:
\begin{equation}\label{eq4.57}
h_a(x) = \tilde{g}_{s_a}(v_a(x)) + \frac{1}{2} c_a |x - x^*|^2 \zeta_a(x).
\end{equation}
Clearly, $h_a(x) \in C^2(\bar{\Omega})$ and $h_a(x) = 0$ on $\partial\Omega.$ 
For $a \geq a_2,$  we construct our candidate subsolution $w_a(x)$ by:
\begin{equation}\label{eq4.58}
w_a(x) = a h_a(x) + \phi^*(x).
\end{equation}
Note that $w_a(x) \in C^2(\bar{\Omega})$ and $w_a(x) = \phi(x)$ on $\partial\Omega.$ 

The gradient and Hessian of $h_a(x)$ and $w_a(x)$ are given by:
\begin{equation}\label{eq4.59}
D h_a(x) = \tilde{g}'_{s_a}(v_a(x)) D v_a(x) + c_a (x - x^*) \zeta_a(x) + \frac{1}{2} c_a |x - x^*|^2 D \zeta_a(x),
\end{equation}
\begin{align}\label{eq4.60}
D^2 h_a(x) &= \tilde{g}''_{s_a}(v_a(x)) D v_a(x) \otimes D v_a(x) + \tilde{g}'_{s_a}(v_a(x)) D^2 v_a(x) \nonumber \\
&\quad + c_a E_n \zeta_a(x) + c_a (x - x^*) \otimes D \zeta_a(x) + \frac{1}{2} c_a |x - x^*|^2 D^2 \zeta_a(x),
\end{align}
\begin{equation}\label{eq4.61}
D w_a(x) = a \left[ D h_a(x) + \frac{1}{a} D\phi^*(x) \right],
\end{equation}
\begin{equation}\label{eq4.62}
D^2 w_a(x) = a \left[ D^2 h_a(x) + \frac{1}{a} D^2\phi^*(x) \right],
\end{equation}
where $E_n$ is the $n\times n$ identity matrix.

Our goal is to prove that under the assumptions of Theorem \ref{thm1.2}, the function $w_a(x)$ is an admissible function for equation \eqref{eq1.1}. Furthermore, under the growth conditions on $f(x, z, p),$  $w_a(x)$ serves as an admissible subsolution to the Dirichlet problem \eqref{eq1.1}, \eqref{eq1.2}.

\subsection{Admissibility of $w_a(x)$ in $\Omega \setminus \Omega_{\delta_1(a)}$}

In the interior region $\Omega \setminus \Omega_{\delta_1(a)} = \Omega \setminus \Omega_{1/a^{\gamma_1}},$  we have:
$v_a(x) = -\frac{1}{5} s_a,$  $\tilde{g}'_{s_a}(v_a(x)) = 0,$  $\zeta_a(x) = 1,$  and $\phi^*(x) = 0.$ 
Hence:
\begin{equation}\label{eq4.63}
w_a(x) = a \left[ -\frac{1}{5} s_a + \frac{1}{2} c_a |x - x^*|^2 \right],
\end{equation}
\begin{equation}\label{eq4.64}
D w_a(x) = a c_a (x - x^*) = a^{1 - \gamma_4} (x - x^*),
\end{equation}
\begin{equation}\label{eq4.65}
D^2 w_a(x) = a^{1 - \gamma_4} E_n.
\end{equation}

Therefore, the augmented Hessian matrix satisfies:
\begin{align}\label{eq4.66}
\omega(w_a)(x) &= a \left[ c_a E_n - \frac{1}{a} A(x, w_a(x), D w_a(x)) \right] \nonumber \\
&= a \left[ c_a E_n - \frac{1}{a} A_1(x, w_a, D w_a) - \frac{1}{a} A_2(x, w_a, D w_a) \right] \nonumber \\
&=: a \omega_3(x) - A_2(x, w_a, D w_a),
\end{align}
where
\begin{equation}\label{eq4.67}
\omega_3(x) = c_a E_n + \alpha_3(x),
\end{equation}
\begin{equation}\label{eq4.68}
\alpha_3(x) = -\frac{1}{a} A_1(x, w_a(x), D w_a(x)).
\end{equation}

Since $s_a > \frac{1}{a^{\gamma_3}} \alpha_1(a),$  $c_a = \frac{1}{a^{\gamma_4}},$  and $0 < \gamma_3 < \gamma_4,$  from \eqref{eq4.63} and \eqref{eq4.64} there exists $a_3 > a_2$ such that for all $a \geq a_3:$ 
\begin{equation}\label{eq4.69}
-\frac{1}{4} a^{1 - \gamma_3} < w_a(x) \leq -\frac{1}{6} a^{1 - \gamma_3} < 0, \quad x \in \Omega \setminus \Omega_{\delta_1(a)},
\end{equation}
\begin{equation}\label{eq4.70}
|D w_a(x)| \leq d a^{1 - \gamma_4}, \quad 1 \leq d a^{1 - \gamma_4},
\end{equation}
where $d = d(\Omega)$ is the diameter of $\Omega.$ 

Since $A_1(x, z, p)$ has growth order $(C_0,\chi, \gamma)$ in $\Omega \setminus \Omega_{\delta_1(a)}$ for $z \leq 0$ with $0 < \chi< 1, 0 < \gamma \leq 1,$  applying \eqref{eq4.69} and \eqref{eq4.70} yields for $x \in \Omega \setminus \Omega_{\delta_1(a)}:$ 
\begin{align*}
|\alpha_3(x)| &= \left| \frac{1}{a} A_1(x, w_a(x), D w_a(x)) \right| \\
&\leq \frac{2 C_0}{a} (1 + |w_a(x)|)^{-\chi} (1 + |D w_a(x)|)^\gamma \\
&\leq \frac{2 C_0}{a} \left( \frac{1}{6} a^{1 - \gamma_3} \right)^{-\chi} (1 + d a^{1 - \gamma_4})^\gamma \\
&\leq C a^{- \chi(1 - \gamma_3) + (\gamma-1)(1 - \gamma_4)}\frac{1}{a^{\gamma_4}}.
\end{align*}
Since $0 < \chi< 1,$  $1 - \gamma_3 > 0,\gamma\leq 1,$  we have $- \chi(1 - \gamma_3) +(\gamma-1)(1-\gamma_4)< 0$ and there exists $a_4 > a_3$ such that for all $a \geq a_4:$ 
\[
|\alpha_3(x)| \leq C a^{-1 - \chi(1 - \gamma_3)} \leq \frac{1}{3} a^{-\gamma_4}.
\]
From \eqref{eq4.67}, it follows that $\omega_3(x) > 0$ and thus $S_k(\omega_3(x)) > 0$ in $\Omega \setminus \Omega_{\delta_1(a)}.$  Since $A_2(x, w_a, D w_a) \leq 0,$  by assertion 1 of Proposition \ref{prop2.3}, we obtain $S_k(a \omega_3(x) - A_2(x, w_a, D w_a)) > 0.$  Hence, $w_a(x)$ is admissible in $\Omega \setminus \Omega_{\delta_1(a)}.$ 

\subsection{Matrix Decomposition of $\omega(w_a)(x)$ in the Boundary Neighborhood $\Omega_{\delta_1(a)}$}

For $x \in \Omega_{\delta_1(a)} = \Omega_{1/a^{\gamma_1}},$  let $x_0 \in \partial\Omega$ be the unique nearest boundary point with $|x - x_0| = d(x).$  Using \eqref{eq4.59}--\eqref{eq4.62}, we decompose $\omega(w_a)(x)$ as follows:
\begin{align*}
\omega(w_a)(x) =& D^2 w_a(x) - A(x, w_a(x), D w_a(x)) \nonumber \\
=&a\Bigg[\tilde{g}_{s_a}''(v_a(x))Dv_a(x)\otimes Dv_a(x)+g_{s_a}'(v_a(x))D^2v_a(x)+c_aE_n\zeta_a(x)\nonumber \\
&\hspace{1cm}+2c_a(x-x^*)\otimes D\zeta_a(x)+\frac 1 2 c_a|x-x^*|^2 D^2\zeta_a(x)+\frac{1}{a}D^2\varphi^*(x)\nonumber \\
&\hspace{1cm}-\frac{1}{a}A_1(x,w_a,Dw_a)-\frac{1}{a}A_2(x,w_a,Dw_a)\Bigg]\nonumber \\
\end{align*}
\begin{align}
=&a\Bigg[{g}_{s_a}'(v_a(x))\Big[D^2v_a(x_0)-\sum_{j=1}^n D_{p_j}A_1(x_0,0,0)D_{x_j}v_a(x_0)\Big]\nonumber \\
&\hspace{1cm}+g_{s_a}'(v_a(x))\left[D^2v_a(x)-D^2v_a(x)(x_0)\right]\nonumber \\
&\hspace{1cm}+g_{s_a}''(v_a(x))Dv_a(x)\otimes Dv_a(x)+c_aE_n\zeta_a(x)+2c_a(x-x^*)\otimes D\zeta_a(x)\nonumber \\
&\hspace{1cm}+\frac 1 2 c_a|x-x^*|^2 D^2\zeta_a(x)+\frac{1}{a}D^2\zeta_a(x)
+\frac 1 a D^2\varphi^*(x)\nonumber \\
&\hspace{1cm}-\frac 1 a A_2(x,w_a,Dw_a)-\frac 1 a A_1(x,0,0)-\nonumber \\
&\Big[\frac{1}{a}\big(A_1(x,w_a,Dw_a)-A_1(x,0,0)\big)-g_{s_a}'(v_a(x))\sum_{j=1}^nD_{p_j}A_1(x_0,0,0)D_{x_j}v_a(x_0)\Big]\Bigg]\nonumber \\
=&: a \Big[ g'_{s_a}(v_a(x)) \omega_4(x_0) + g'_{s_a}(v_a(x)) \alpha_4(x) + \alpha_5(x) + \omega_5(x) \nonumber \\
&\hspace{1cm} + \alpha_6(x) + \alpha_7(x) + \alpha_8(x) + \alpha_9(x) + \alpha_{10}(x) + B_1(x) \Big].\label{eq4.71}
\end{align}
Applying formula \eqref{eq2.1} to $B_1(x),$  we have:
\begin{align}
B_1(x) =&-\Bigg[\frac 1 a\int_0^1D_zA_1\big(x,tw_a(x),tDw_a(x)\big)w_a(x)dt\nonumber \\
&+\frac 1 a\int_0^1\sum_{j=1}^{n}D_{p_j}A_1\big(x,tw_a(x),tDw_a(x)\big)a\Big[C_a(x_j-x_j^*)\zeta_a(x)\nonumber \\
&\hspace{4cm}+\frac 1 2C_a|x-x^*|^2D_{x_j}\zeta_a(x)+\frac 1 a D_{x_j}\varphi^*(x)\Big]dt\nonumber \\
&+g'_{s_a}(v_a(x))\int_0^1\sum_{j=1}^{n}D_{p_j}A_1\big(x,tw_a(x),tDw_a(x)\big)\big(D_{x_j}v_a(x)-Dv_a(x_0)\big)dt\nonumber \\
&+g'_{s_a}(v_a(x))\int_0^1\sum_{j=1}^{n}\Big(D_{p_j}A_1\big(x,tw_a(x),tDw_a(x)\big)-D_{p_j}A_1(x_0,0,0)Dv_a(x_0)\Big)dt\Bigg]\nonumber\\
=:& \alpha_{11}(x) + \alpha_{12}(x) + g'_{s_a}(v_a(x)) \alpha_{13}(x)+ g'_{s_a}(v_a(x)) B_2(x). \label{eq4.72}
\end{align}
Applying \eqref{eq2.1} again to $B_2(x)$ yields:
\begin{align}
B_2(x) =&-\int_0^1\int_0^1\sum_{i,j=1}^nD^2_{x_ip_j}A_1\big((1-s)x_0+sx,tsw_a(x),tsDw_a(x)\big)(x_i-x_{0i})D_{x_j}v_a(x_0)dtds\nonumber \\
&-\int_0^1\int_0^1\sum_{j=1}^nD^2_{zp_j}A_1\big((1-s)x_0+sx,tsw_a(x),tsDw_a(x)\big)tw_a(x)D_{x_j}v_a(x_0)dtds\nonumber \\
&-\int_0^1\int_0^1\sum_{i,j=1}^nD^2_{p_ip_j}A_1\big((1-s)x_0+sx,tsw_a(x),tsDw_a(x)\big)a\Big[g'_{s_a}(v_a(x))D_{x_i}v_a(x)\nonumber \\
&\hspace{2cm}+C_a(x_i-x_{i}^*)\zeta_a(x)+\frac 1 2 C_a|x-x^*|^2D_{x_i}\zeta_a(x)+\frac 1 a D_{x_i}\varphi^*(x)\Big]dtds\nonumber \\
=:& \alpha_{14}(x) + \alpha_{15}(x) + \alpha_{16}(x).\label{eq4.73}
\end{align}
Combining \eqref{eq4.71}--\eqref{eq4.73}, we obtain the complete representation of $\omega(w_a)(x)$ in $\Omega_{\delta_1(a)}:$ 
\begin{align}\label{eq4.74}
\omega(w_a)(x) &= a \Big[ g'_{s_a}(v_a(x)) \left[\omega_4(x_0) + \alpha_4(x) + \alpha_{13}(x) + \alpha_{14}(x) + \alpha_{15}(x) + \alpha_{16}(x)\right] \nonumber \\
&\hspace{1cm} + \left[\alpha_6(x) + \alpha_7(x) + \alpha_8(x) + \alpha_{10}(x) + \alpha_{11}(x) + \alpha_{12}(x)\right] \nonumber \\
&\hspace{1cm} + \left[\omega_5(x) + \alpha_5(x) + \alpha_9(x)\right] \Big],
\end{align}
where $\omega_5(x) \geq 0, \alpha_5(x) \geq 0, \alpha_9(x) \geq 0$ for all $x \in \Omega_{\delta_1(a)}.$ 

\subsection{Key Lemmas and Norm Estimates at Boundary Neighborhood}

The two leading matrices in expression \eqref{eq4.74} are $g'_{s_a}(v_a(x)) \omega_4(x_0)$ and $\omega_5(x).$  The following lemma provides that $\lambda(\omega_4(x_0))\in \Gamma_k, x_0\in \partial\Omega$ and some other properties. 

\begin{lemma}\label{lem4.1}
Let $v_a(x)$ be defined by \eqref{eq4.10}. For $x_0 \in \partial\Omega,$  set:
\begin{equation}\label{eq4.75}
\omega_4(v_a)(x_0) = D^2 v_a(x_0) - \sum_{j=1}^n D_{p_j} A_1(x_0, 0, 0) D_{x_j} v_a(x_0).
\end{equation}
Suppose the domain $\Omega$ is uniformly $(k-1)$-$A_1$-convex. Then there exist $a_5 > a_4$ and constants $C_{12} > 0,$  $C_{13} > 0$ such that for all $a \geq a_5,$  $1 \leq m < k,$  and $x_0 \in \partial\Omega:$ 
\begin{equation}\label{eq4.76'}
S_k(\omega_4(v_a)(x_0)) \geq C_{12} |D v_a(x_0)|^{k}>0,
\end{equation}
\begin{equation}\label{eq4.77}
S_m(\omega_4(v_a)(x_0)) \leq C_{13} |D v_a(x_0)|^{m-k} S_k(\omega_4(v_a)(x_0)).
\end{equation}
\end{lemma}

\begin{proof}
For $x_0 \in \partial\Omega,$  from \eqref{eq4.12}, \eqref{eq4.13}, \eqref{eq4.6}, and \eqref{eq4.7}, we have:
\begin{equation}\label{eq4.78}
D v_a(x_0) = - D\theta_a(x_0) = -\frac{1}{a^{\gamma_2}} D d(x_0) = -\frac{1}{a^{\gamma_2}} e'_n J(x_0)^T,
\end{equation}
\begin{align}\label{eq4.79}
D^2 v_a(x_0) &= t_a D\theta_a(x_0) \otimes D\theta_a(x_0) - D^2\theta_a(x_0) \nonumber \\
&= J(x_0) \left[ \frac{t_a}{a^{2\gamma_2}} e'_n \otimes e'_n - \frac{1}{a^{\gamma_2}} \operatorname{diag}(-\kappa_1(x_0), \dots, -\kappa_{n-1}(x_0), 0) \right] J(x_0)^T \nonumber \\
&= \frac{1}{a^{\gamma_2}} J(x_0) \operatorname{diag}\left(\kappa_1(x_0), \dots, \kappa_{n-1}(x_0), \frac{t_a}{a^{\gamma_2}}\right) J(x_0)^T.
\end{align}
On the other hand:
\begin{equation}\label{eq4.80}
\sum_{j=1}^n D_{p_j} A_1(x_0, 0, 0) D_{x_j} v_a(x_0) = |D v_a(x_0)| \sum_{j=1}^n D_{p_j} A_1(x_0, 0, 0) \nu_j(x_0).
\end{equation}
Substituting \eqref{eq4.79} and \eqref{eq4.80} into \eqref{eq4.75} yields:
\begin{align}\label{eq4.81}
\omega_4(v_a)(x_0) &= \frac{1}{a^{\gamma_2}} J(x_0) \tilde{\omega}_4(x_0) J(x_0)^T,
\end{align}
where
\begin{equation}\label{eq4.82}
\tilde{\omega}_4(x_0) = \begin{bmatrix} M(x_0, A_1) & \begin{matrix} \beta_1 \\ \vdots \\ \beta_{n-1} \end{matrix} \\[1em] \begin{matrix} \beta_1 & \dots & \beta_{n-1} \end{matrix} & \frac{t_a}{a^{\gamma_2}} + \beta_n \end{bmatrix},
\end{equation}
and $\beta_1, \dots, \beta_{n-1}, \beta_n$ are quantities uniformly bounded with respect to $x_0 \in \partial\Omega.$ 

From \eqref{eq4.81} and \eqref{eq4.82} we have 
\begin{equation}\label{eq4.83}
S_k(\omega_4(v_a)(x_0)) = S_k(\tilde{\omega}_4(x_0)) = \left(\frac{1}{a^{\gamma_2}}\right)^k \left[ S_{k-1}(M(x_0, A_1)) \left(\frac{t_a}{a^{\gamma_2}} + \beta_n\right) + \tilde{P}_k \right],
\end{equation}
\begin{equation}\label{eq4.84}
S_m(\omega_4(v_a)(x_0)) = S_m(\tilde{\omega}_4(x_0)) = \left(\frac{1}{a^{\gamma_2}}\right)^m \left[ S_{m-1}(M(x_0, A_1)) \left(\frac{t_a}{a^{\gamma_2}} + \beta_m\right) + \tilde{P}_m \right],
\end{equation}
for $m = 1, \dots, k-1,$  where $\tilde{P}_1, \dots, \tilde{P}_k$ are uniformly bounded terms.

Since $t_a = a^{\gamma_5}$ with $\gamma_5 > \gamma_2$ and $\Omega$ is uniformly $(k-1)$-$A_1$-convex, there exist constants $C^{(m)}_{15} \geq C^{(m)}_{14} > 0$ such that for all $a \geq a_5:$ 
\begin{equation}\label{eq4.85}
0 < C^{(m)}_{14} \leq S_{m-1}(M(x_0, A_1)) \leq C^{(m)}_{15}, \quad m = 1, \dots, k.
\end{equation}
Since $\frac{t_a}{a^{\gamma_2}} = a^{\gamma_5 - \gamma_2} \to +\infty$ as $a \to \infty,$  from \eqref{eq4.83}--\eqref{eq4.85} there exist constants $C_{17} > C_{16} > 0$ such that for all $a \geq a_5:$ 
\begin{equation}\label{eq4.86}
\left(\frac{1}{a^{\gamma_2}}\right)^k C_{16} \left(\frac{t_a}{a^{\gamma_2}}\right) < S_k(\omega_4(v_a)(x_0)) < C_{17} \left(\frac{t_a}{a^{\gamma_2}}\right) \left(\frac{1}{a^{\gamma_2}}\right)^k,
\end{equation}
\begin{equation}\label{eq4.87}
\left(\frac{1}{a^{\gamma_2}}\right)^m C_{16} \left(\frac{t_a}{a^{\gamma_2}}\right) < S_m(\omega_4(v_a)(x_0)) < C_{17} \left(\frac{t_a}{a^{\gamma_2}}\right) \left(\frac{1}{a^{\gamma_2}}\right)^m, \quad m=1, \dots, k-1.
\end{equation}
Noting that $|Dv_a(x_0)| = \frac{1}{a^{\gamma_2}},$  dividing \eqref{eq4.86} by \eqref{eq4.87} yields the estimates \eqref{eq4.76'}--\eqref{eq4.77} with $C_{12} = \frac{C_{16}}{C_{17}}$ and $C_{13} = \frac{C_{17}}{C_{16}}.$ From \eqref{eq4.74} and the connectedness of $\partial\Omega$ it follows that $\lambda(\omega_4(v_a(x_0)))\in \Gamma_k, x_0\in \partial\Omega.$ 
\end{proof}

\section{Illustrative Examples}

In this section, we present concrete examples of symmetric augmented matrices $A(x, z, p),$  right-hand side functions $f(x, z, p),$  and domains $\Omega$ that satisfy all structural conditions $(\mathcal{C}_{\delta_0})$ of Theorem \ref{thm1.2}, thereby guaranteeing the existence of admissible subsolutions to the Dirichlet problem \eqref{eq1.1}--\eqref{eq1.2}.

In particular, Example \ref{ex5.2} demonstrates that Theorem \ref{thm1.2} applies even to domains that are \textbf{not} $(k-1)$-convex in the classical sense, highlighting the fundamental geometric significance of the augmented matrix term $A_1(x, z, p).$ 

\begin{example}\label{ex5.1}
Let $\Omega = \{ x \in \mathbb{R}^n: |x| < 1 \}$ be the unit ball in $\mathbb{R}^n.$  The boundary $\partial\Omega = \{ |x| = 1 \}$ is smooth ($C^\infty$) and its principal curvatures with respect to the unit inner normal vector $-\nu(x_0) = -x_0$ are $\kappa_1(x_0) = \dots = \kappa_{n-1}(x_0) = 1$ for all $x_0 \in \partial\Omega.$ 

Consider the augmented matrix decomposed as $A(x, z, p) = A_1(x, z, p) + A_2(x, z, p),$  where:
\begin{equation}\label{eq5.1}
A_1(x, z, p) = - \frac{(1 - z)^{-\chi}}{(1 + |x|^2)^\beta (1 + |p|^2)^{(1-\gamma)/2}} E_n \quad \text{for } z \leq 0,
\end{equation}
\begin{equation}\label{eq5.2}
A_2(x, z, p) = -\lambda_0 (-z)^{-\chi} (1 + |p|^2)^{(\gamma-2)/2} E_n \quad \text{for } z \leq 0,
\end{equation}
where $E_n$ denotes the $n \times n$ identity matrix, $\lambda_0 \geq 0$ is a non-negative constant, $0 < \chi< 1,$  $0 < \gamma \leq 1,$  and $\beta \geq 0.$ 

Let $f(x, z, p) = e^{-z} (1 + |p|^2)^{l/2}$ for $z \leq 0,$  where $0 \leq l < k.$  We verify that all structural hypotheses of Theorem \ref{thm1.2} are satisfied:
\begin{enumerate}
\item \textbf{Negative semi-definiteness of $A_2$}: Since $\lambda_0 \geq 0$ and $z \leq 0,$  we have $(-z)^{-\chi} \geq 0,$  which implies $A_2(x, z, p) \leq 0$ for all $z \leq 0.$  Thus condition (2) holds.
\item \textbf{Uniform $(k-1)$-$A_1$-convexity of $\Omega$}: At $x_0 \in \partial\Omega$ and $z=0, p=0,$  we have $A_1(x_0, 0, 0) = -\frac{1}{(1 + |x_0|^2)^\beta} E_n = -\frac{1}{2^\beta} E_n$ and $D_{p_j} A_1(x_0, 0, 0) = 0$ for all $j=1, \dots, n.$  Consequently, the characteristic boundary matrix defined in \eqref{eq1.4} reduces to:
\[
M(x_0, A_1) = \operatorname{diag}(\kappa_1(x_0), \dots, \kappa_{n-1}(x_0)) = \operatorname{diag}(1, \dots, 1) \in \mathbb{R}^{(n-1) \times (n-1)}.
\]
Therefore, $S_{k-1}(M(x_0, A_1)) = \sigma_{k-1}(1, \dots, 1) = \binom{n-1}{k-1} > 0.$  This confirms that $\Omega$ is uniformly $(k-1)$-$A_1$-convex, satisfying condition (3).
\item \textbf{Growth order of $A_1$ in $\bar{\Omega} \setminus \Omega_\delta$}: For $z \leq 0,$  we have $1 - z = 1 + |z|.$  Thus:
\[
|A_1(x, z, p)| \leq \frac{(1 + |z|)^{-\chi}}{(1 + |p|^2)^{(1-\gamma)/2}} \leq C_0 (1 + |z|)^{-\chi} (1 + |p|)^\gamma,
\]
which shows that $|A_1(x, z, p)| \frac{(1+|z|)^\chi }{(1+|p|)^\gamma} \leq C_0,$  satisfying condition (4) with growth order $(C_0,\chi, \gamma)$ for $0 < \chi< 1$ and $0 < \gamma \leq 1.$ 
\item \textbf{Derivative growth orders in $\bar{\Omega}_\delta$}:
\begin{enumerate}
\item Differentiating $A_1$ with respect to $z:$ 
\[
D_z A_1(x, z, 0) = \frac{\chi (1 - z)^{-\chi-1}}{(1 + |x|^2)^\beta} E_n \geq 0 \quad \text{for } z \leq 0,
\]
since $\chi  > 0$ and $1 - z \geq 1,$  satisfying condition (5.a).
\item Differentiating $A_1$ with respect to $p_j:$ 
\[
D_{p_j} A_1(x, z, p) = -\frac{(1 - z)^{-\chi}}{(1 + |x|^2)^\beta} \frac{-(1-\gamma) p_j}{(1 + |p|^2)^{(3-\gamma)/2}} E_n.
\]
Since $1 - z = 1 + |z|$ and $\frac{|p_j|}{(1 + |p|^2)^{(3-\gamma)/2}} \leq (1 + |p|)^{\gamma-1},$  we obtain:
\[
|D_p A_1(x, z, p)| \leq C (1 + |z|)^{-\chi} (1 + |p|)^{\gamma-1},
\]
satisfying condition (5.b) with growth order $(*,\chi, \gamma-1)$ for $0 < \chi< 1.$ 
\item Differentiating $A_1$ with respect to $x_i$ and $p_j:$ 
\[
D^2_{x_i p_j} A_1(x, z, p) = -\frac{-2\beta x_i (1 - z)^{-\chi}}{(1 + |x|^2)^{\beta+1}} \frac{-(1-\gamma) p_j}{(1 + |p|^2)^{(3-\gamma)/2}} E_n.
\]
We have:
\[
|D^2_{xp} A_1(x, z, p)| \leq C_1 (1 + |z|)^{-\chi} (1 + |p|)^{\gamma-1},
\]
satisfying condition (5.c) with growth order $(C_1,\chi, \gamma-1)$ for $0 < \chi< 1.$ 
\item Differentiating $A_1$ with respect to $z$ and $p_j:$ 
\[
D^2_{z p_j} A_1(x, z, p) = \frac{\chi (1 - z)^{-\chi-1}}{(1 + |x|^2)^\beta} \frac{-(1-\gamma) p_j}{(1 + |p|^2)^{(3-\gamma)/2}} E_n.
\]
For $z \leq 0,$  we have $(1 - z)^{-\chi-1} = (1 + |z|)^{-(\chi+1)}.$  Thus:
\[
|D^2_{zp} A_1(x, z, p)| \leq C (1 + |z|)^{-(\chi+1)} (1 + |p|)^{\gamma-1},
\]
satisfying condition (5.d) with growth order $(*,\chi+1, \gamma-1)$ precisely.
\end{enumerate}
\item \textbf{Growth order of $f$}: The function $f(x, z, p) = e^{-z} (1 + |p|^2)^{l/2}$ has growth order $(C_2, 0, l)$ with $0 \leq l < k,$  satisfying condition (6).
\end{enumerate}
By Theorem \ref{thm1.2}, there exists an admissible subsolution $u(x)$ to the Dirichlet problem \eqref{eq1.1}--\eqref{eq1.2}.
\end{example}

\begin{example}\label{ex5.2}
Let $n \geq 3$ and $2 \leq k \leq n.$  Consider the solid torus domain in $\mathbb{R}^n:$ 
\[
\Omega = \left\{ x = (x_1, \dots, x_n) \in \mathbb{R}^n: \left(\sqrt{x_1^2 + x_2^2} - R\right)^2 + x_3^2 + \dots + x_n^2 < r^2 \right\},
\]
where $0 < r < R$ are fixed radii.

The boundary $\partial\Omega = \left\{ \left(\sqrt{x_1^2 + x_2^2} - R\right)^2 + x_3^2 + \dots + x_n^2 = r^2 \right\} \cong S^1 \times S^{n-2}$ is a smooth, connected $(n-1)$-dimensional hypersurface in $\mathbb{R}^n.$ 

At any boundary point $x_0 \in \partial\Omega,$  let $u \in [0, 2\pi)$ denote the poloidal angle such that $\sqrt{x_{0,1}^2 + x_{0,2}^2} = R + r \cos u.$  The principal curvatures of $\partial\Omega$ with respect to the unit inner normal vector $-\nu(x_0)$ pointing into $\Omega$ are:
\[
\kappa_1(x_0) = \dots = \kappa_{n-2}(x_0) = \frac{1}{r} > 0, \quad \kappa_{n-1}(x_0) = \frac{\cos u}{R + r \cos u}.
\]
On the inner region of the torus where $\cos u < 0$ (facing the central hole), the principal curvature $\kappa_{n-1}(x_0) < 0.$  In particular, at the innermost equator ($\cos u = -1$), we have:
\[
\kappa_{n-1}(x_0) = -\frac{1}{R - r} < 0.
\]
Consequently, $S_{k-1}(\kappa_1, \dots, \kappa_{n-1})$ fails to be positive on the inner boundary region for $k \geq 2, k$ is even.  Thus, the domain $\Omega$ is not $(k-1)$-convex in the classical sense of Caffarelli--Nirenberg--Spruck \cite{3}.

Now, let $\Phi(x) = \left(\sqrt{x_1^2 + x_2^2} - R\right)^2 + x_3^2 + \dots + x_n^2 - r^2,$  so that $\Omega = \{ \Phi(x) < 0 \}$ and $D\Phi(x_0) = 2r \nu(x_0)$ on $\partial\Omega,$  where $\nu(x_0)$ is the unit outward normal vector. Set $V(x) = \frac{\mu_0}{2r} D\Phi(x) \in C^2(\bar{\Omega}, \mathbb{R}^n),$  where $\mu_0 > 0$ is a positive constant, so that $V(x_0) = \mu_0 \nu(x_0)$ on $\partial\Omega.$ 

Consider the augmented matrix $A(x, z, p) = A_1(x, z, p) + A_2(x, z, p)$ defined for $z \leq 0$ by:
\begin{equation}\label{eq5.3}
A_1(x, z, p) = \frac{\langle V(x), p \rangle}{(1 + |x|^2)^\beta (1 + |p|^2)^{(1-\gamma)/2} (1 - z)^\chi } E_n,
\end{equation}
\begin{equation}\label{eq5.4}
A_2(x, z, p) = -\lambda_0 (-z)^{-\chi} (1 + |p|^2)^{(\gamma-2)/2} E_n,
\end{equation}
where $\mu_0 > 0, \lambda_0 \geq 0, 0 < \chi< 1, 0 < \gamma \leq 1, \beta \geq 0.$ 

At $x_0 \in \partial\Omega$ and $z=0, p=0,$  we have $A_1(x_0, 0, 0) = 0$ and $D_{p_d} A_1(x_0, 0, 0) = \frac{\mu_0 \nu_d(x_0)}{(1 + |x_0|^2)^\beta} E_n.$  Consequently:
\[
\sum_{d=1}^n D_{p_d} A_1(x_0, 0, 0) \nu_d(x_0) = \frac{\mu_0 \sum_{d=1}^n \nu_d^2(x_0)}{(1 + |x_0|^2)^\beta} E_n = \frac{\mu_0}{(1 + |x_0|^2)^\beta} E_n.
\]
The characteristic boundary matrix $M(x_0, A_1)$ defined in \eqref{eq1.4} becomes:
\[
M(x_0, A_1) = \operatorname{diag}\left( \kappa_1(x_0) + \frac{\mu_0}{(1 + |x_0|^2)^\beta}, \dots, \kappa_{n-1}(x_0) + \frac{\mu_0}{(1 + |x_0|^2)^\beta} \right).
\]
Since $|x_0| \leq R + r$ for all $x_0 \in \partial\Omega,$  we have $(1 + |x_0|^2)^\beta \leq (1 + (R+r)^2)^\beta =: C_R.$ 
By choosing $\mu_0 > \frac{C_R}{R - r}$ (for example, $\mu_0 = \frac{2 C_R}{R - r}$), the smallest eigenvalue of $M(x_0, A_1)$ satisfies:
\[
\lambda_{n-1}(M(x_0, A_1)) = \kappa_{n-1}(x_0) + \frac{\mu_0}{(1 + |x_0|^2)^\beta} \geq -\frac{1}{R - r} + \frac{\mu_0}{C_R} > 0.
\]
Thus, ALL eigenvalues of $M(x_0, A_1)$ are strictly positive at every point $x_0 \in \partial\Omega.$  Therefore, $\lambda(M(x_0, A_1)) \in \Gamma_{n-1} \subset \Gamma_{k-1},$  and $S_{k-1}(M(x_0, A_1)) \geq C_1 > 0$ on $\partial\Omega.$ 

This proves that the solid torus $\Omega$ is uniformly $(k-1)$-$A_1$-convex, even though its boundary $\partial\Omega$ is connected and $\Omega$ is not $(k-1)$-convex in the classical sense. All growth conditions (4)--(6) are verified similarly to Example \ref{ex5.1}. Hence, Theorem \ref{thm1.2} guarantees the existence of an admissible subsolution $u(x).$ 
\end{example}

\begin{example}\label{ex5.3}
Let $\Omega = \{ x \in \mathbb{R}^n: \sum_{j=1}^n \mu_j^2 x_j^2 < 1 \}$ ($\mu_j > 0$) be an ellipsoid in $\mathbb{R}^n.$ 
Let $A(x, z, p) = A_1(x, z, p) + A_2(x, z, p)$ with:
\[
A_1(x, z, p) = \frac{g(x) (1-z)^{-\chi}}{(1 + |p|^2)^\theta} P(x) \quad \text{for } z \leq 0,
\]
\[
A_2(x, z, p) = -c_0 (1 + |p|^2)^{(\gamma-2)/2} (1-z)^{-\chi} E_n \quad \text{for } z \leq 0,
\]
where $g(x) \geq 0$ is a smooth non-negative function, $P(x)$ is a smooth positive semi-definite matrix field on $\bar{\Omega},$  $c_0 \geq 0,$  $\theta = (1-\gamma)/2 \geq 0,$  $0 < \chi< 1,$  and $0 < \gamma \leq 1.$ 

When the norm $\|P\|_{C^2(\bar{\Omega})}$ is sufficiently small in the boundary neighborhood $\Omega_\delta,$  the perturbed curvature matrix $M(x_0, A_1)$ remains in the Garding cone $\Gamma_{k-1}$ for all $x_0 \in \partial\Omega.$  Thus, the domain $\Omega$ is uniformly $(k-1)$-$A_1$-convex, and Theorem \ref{thm1.2} guarantees the existence of an admissible subsolution.
\end{example}

\end{document}